\documentclass[11pt, reqno]{amsart}

\usepackage[utf8]{inputenc}
\usepackage[T2A]{fontenc}
\usepackage[english,ukrainian,russian]{babel}

\usepackage[english]{d-omega3}

\usepackage{amssymb}
\usepackage{graphicx}
\usepackage{xcolor}
\usepackage[all]{xy}
\usepackage{multicol}

\usepackage{graphicx}
\usepackage{mathtools}
\usepackage{amssymb}
\usepackage{amsmath, amsthm}
\usepackage{amscd}
\usepackage{enumitem}
\usepackage[all]{xy}

\usepackage{hyperref}
\hypersetup{hidelinks=true, hypertexnames=false}
\hypersetup{
    hypertexnames=false,
    colorlinks,
    linkcolor={red!30!black},
    citecolor={blue!50!black},
    urlcolor={blue!80!black}
}

\newcommand\bC{{\mathbb C}}
\newcommand\DDD{{\mathbb D}}
\newcommand\bN{{\mathbb N}}

\newcommand\bR{{\mathbb R}}

\newcommand\bZ{{\mathbb Z}}

\newcommand\FF{{\mathcal F}}

\newcommand\KK{{\mathcal K}}

\newcommand\NN{{\mathcal N}}

\newcommand\cov[1]{\tilde{#1}}
\newcommand\eps{\varepsilon}

\newcommand\id{\mathrm{id}}          
\newcommand\Int{\mathrm{Int}\,}        

\newcommand\Aman{A}
\newcommand\Bman{B}

\newcommand\Dman{D}
\newcommand\Eman{E}
\newcommand\Kman{K}
\newcommand\Lman{L}
\newcommand\Mman{M}
\newcommand\Nman{N}
\newcommand\Pman{P}
\newcommand\Qman{Q}
\newcommand\Rman{R}

\newcommand\Tman{T}
\newcommand\Uman{U}
\newcommand\Vman{V}
\newcommand\Wman{W}
\newcommand\Xman{X}
\newcommand\Yman{Y}
\newcommand\Zman{Z}

\newcommand\Circle{S^1}            

\newcommand\SO{\mathrm{SO}}

\newcommand\Orb{\mathcal{O}}        
\newcommand\Stab{\mathcal{S}}       
\newcommand\Diff{\mathcal{D}}       
\newcommand\Homeo{\mathcal{H}}      

\newcommand\StabId{\Stab_{\id}}     

\newcommand\Cr[1]{\mathcal{C}^{#1}}
\newcommand\Cinfty{\Cr{\infty}}
\newcommand\Crm[3]{\mathcal{C}^{#1}(#2,#3)}               
\newcommand\Cont[2]{\Crm{}{#1}{#2}}            
\newcommand\Ci[2]{\Crm{\infty}{#1}{#2}}               

\newcommand\func{f}
\newcommand\gfunc{g}
\newcommand\dif{h}
\newcommand\hdif{\cov{\dif}}
\newcommand\gdif{g}

\newcommand\qdif{q}

\newcommand\Morse{\mathrm{Morse}}

\newcommand\Stabilizer[1]{\Stab(#1)}             
\newcommand\StabilizerPlus[1]{\Stab^{+}(#1)}     
\newcommand\StabilizerMinus[1]{\Stab^{-}(#1)}     
\newcommand\StabilizerId[1]{\StabId(#1)}         

\newcommand\Orbit[1]{\Orb(#1)}                   

\newcommand\SingularSet[1]{\Sigma_{#1}}             

\newcommand\AutKRGraphStab[1]{\mathbf{G}}             

\newcommand\fSing{\SingularSet{\func}}                

\newcommand\regN[1]{R_{#1}}

\newcommand\crLev{K}

\newcommand\fld{F}
\newcommand\flow{\mathbf{F}}
\newcommand{\hflow}{\cov{\flow}}

\newcommand\gfld{G}
\newcommand\gflow{\mathbf{G}}

\newcommand\hfld{\cov{\fld}}

\newcommand\FolStab{\Delta} 

\newcommand\MFolStabilizer[1]{\FolStab^{-}(#1)}

\newcommand\ddx[2]{\tfrac{\partial#1}{\partial#2}}

\newcommand\halpha{\cov{\alpha}}

\newcommand{\tXman}{\tilde{\Xman}}

\newcommand\KRGraphf{\Gamma_{\func}}

\begin{document}

\newcommand{\comment}[1]{}

\newcommand{\kerr}{\operatorname{Ker}}
\newcommand{\aut}{\operatorname{Aut}}
\newcommand\St[1]{\mathcal{S}(#1)}
\newcommand\Sfdd{\St{f, \partial D^2}}
\newcommand\SfdM{\St{f, \partial\Mman}}
\newcommand\SfdMX{\St{f, \partial\Mman;X}}
\newcommand{\DD}[2]{D_{#1}^{#2}}
\newcommand{\ova}{\overline{A_i}}
\newcommand{\po}{\varphi_1}
\newcommand{\pot}{\varphi_2}

\newcommand{\Com}{\left[ G\mathop{\wr}\limits_n\bZ, G\mathop{\wr}\limits_n\bZ \right]}
\newcommand{\com}{\left[ G\wr_n\bZ, G\wr_n\bZ \right]}

\newcommand{\pd}{\partial\Mman}
\newcommand{\ha}{\widehat{A}}
\newcommand{\gha}{\widehat{g}}
\newcommand{\delm}{\Delta^-(f)}
\newcommand{\eks}{(\Circle\times N)\setminus L}
\newcommand{\wh}{\hdif}
\newcommand{\n}{(\bR\times n)\setminus\widetilde{L}}
\newcommand\invp[1]{({#1},+)}
\newcommand\invm[1]{({#1},-)}

\newcommand\sq{\mathsf{sq}}
\newcommand\shift{\sigma}

\newcommand\CrCircle[2]{C^{#1}_{#2}(\Circle,\Circle)}
\newcommand\CCircle[1]{\CrCircle{}{#1}}
\newcommand\pz{z}
\newcommand\px{x}
\newcommand\pt{t}
\newcommand\ps{s}
\newcommand\py{y}
\newcommand\prj{p}
\newcommand\tnp{TNP}
\newcommand\tnpz[1]{$\tnp(#1)$}

\newcommand\regNV{\Rman_{\Vman}}
\newcommand\regZ{\regN{\Zman}}

\newcommand\FMP{\mathcal{F}(\Mman,\Pman)}

\newcommand\tLman{\widetilde{\Lman}}
\newcommand\tYman{\widetilde{\Yman}}
\newcommand\Gx{L_{\px}}
\newcommand\tGx{\widetilde{L}_{\px}}

\newcommand\vk{k}
\newcommand\vl{l}
\newcommand\chp{\NN}
\newcommand\SingFol{\Xi_{\func}}
\newcommand\SGman{\Tman}
\newcommand\str[1]{\sigma_{#1}}
\newcommand\bs[2]{\mathfrak{b}_{#1}(#2)}

\newcommand\ori{\mathsf{or}}
\newcommand\minv[1]{#1^{-}}
\newcommand\pinv[1]{#1^{+}}

\newcommand\HZ{\Homeo_{cf}(\Zman)}
\newcommand\HY{\Homeo_{cf}(\Yman)}
\newcommand\jind{{\color{red}j}}

\newcommand\Ji[1]{I_{#1}}
\newcommand\tJi[1]{\widetilde{I}_{#1}}

\newcommand\tHsp{\mathbb{H}} 
\newcommand\oHsp{\Int(\mathbb{H})} 

\newcommand\chpmm{\chp^{-}}
\newcommand\chppp{\chp^{+}}
\newcommand\Kmm{\Kman^{-}}
\newcommand\Kpp{\Kman^{+}}
\newcommand\Rmm{\regN{\Kmm}}
\newcommand\Rpp{\regN{\Kpp}}

\title{Reversing orientation homeomorphisms of surfaces}
\author{Iryna Kuznietsova, Sergiy Maksymenko}
\address{Topology Laboratory, Department of Algebra and Topology, Institute of Mathematics of NAS of Ukraine, Tereshchenkivska str. 3, Kyiv, 01601, Ukraine}
\email{maks@imath.kiev.ua, kuznietsova@imath.kiev.ua}


\abstract{english}{Let $M$ be a connected compact orientable surface, $f:M\to \mathbb{R}$ be a Morse function, and $h:M\to M$ be a diffeomorphism which preserves $f$ in the sense that $f\circ h = f$.
We will show that if $h$ leaves invariant each regular component of each level set of $f$ and reverses its orientation, then $h^2$ is isotopic to the identity map of $M$ via $f$-preserving isotopy.

\apar
This statement can be regarded as a foliated and a homotopy analogue of a well known observation that every reversing orientation orthogonal isomorphism of a plane has order $2$, i.e. is a mirror symmetry with respect to some line.
The obtained results hold in fact for a larger class of maps with isolated singularities from connected compact orientable surfaces to the real line and the circle.
}

\shortAuthorsList{I.~Kuznietsova, S.~Maksymenko}


\maketitle

\section{Introduction}
The present paper describes several \myemph{foliated} and \myemph{homotopy} variants of a ``rigidity'' property for reversing orientation linear motions of the plane claiming that \myemph{every such motion has order $2$}.
Though it is motivated by study of deformations of smooth functions on surfaces (and we prove the corresponding statements), the obtained results seem to have an independent interest.

Let $\DDD_n =\{ r,s \mid r^n = s^2 = 1,\, rs = s r^{-1} \}$ be the \myemph{dihedral} group, i.e. group of symmetries of a right $n$-polygon.
Then each ``reversing orientation'' element is written as $r^k s$ for some $k$ and has order $2$:
\[
    (r^k s)^2 = r^k s r^k s = r^k r^{-k} s s = 1.
\]

More generally, let
$
    SO^{-}(2) =
    \left\{
        \left(\begin{smallmatrix}
            \sin t & \cos t \\
            \cos t & -\sin t
        \end{smallmatrix}
        \right) \mid t\in[0;2\pi)
    \right\}
$
be the \myemph{set} of reversing orientation orthogonal maps of $\bR^2$, i.e.~the adjacent class of $O(2)/SO(2)$ distinct from $SO(2)$.
Then again one easily checks that each element of $SO^{-}(2)$ has order $2$.

Another counterpart of this effect is that every motion of $\bR$ which reverses orientation is given by the formula: $f(x) = a-x$ for some $a\in\bR$, and therefore it has order $2$, i.e. $f(f(x))=x$.

Notice that such a property does not hold for ``non-rigid motions'', like arbitrary homeomorphisms or diffeomorphisms of $\bR$ or $\Circle$.
For example, let $\dif:\bR\to\bR$ be given by $\dif(x)=-x^3$.
It reverses orientation, but $\dif(\dif(x))= -(-x^3)^3=x^9$, and therefore it is not the identity.

Nevertheless, counterparts of the above rigidity effects for homeomorphisms can still be obtained on a ``homotopy'' level.

For instance, let $\Homeo(\Circle)$ be the group of all homeomorphisms of the circle $\Circle$ and $\Homeo^{+}(\Circle)$, (resp. $\Homeo^{-}(\Circle)$), be the subgroup (resp. subset) consisting of homeomorphisms preserving (resp. reversing) orientation.
Endow these spaces with compact open topologies.
Notice that we have a natural inclusion $O(2) \subset \Homeo(\Circle)$, which \myemph{consists of two inclusions} $SO^{-}(2) \subset \Homeo^{-}(\Circle)$ and $SO(2) \subset \Homeo^{+}(\Circle)$ between the corresponding path components.
It is well known and is easy to see $SO^{-}(2)$ (resp. $SO(2)$) is a strong deformation retract of $\Homeo^{-}(\Circle)$ (resp. $\Homeo^{+}(\Circle)$).
This implies that the map $\sq:\Homeo^{-}(\Circle) \to \Homeo^{+}(\Circle)$ defined by $\sq(\dif) = \dif^2$ is null homotopic.

The aim of the present paper is to prove a parametric variant of the above ``rigidity'' statements for self-homeomorphisms of open subsets of topological products $\Xman\times\Circle$ preserving first coordinate (Theorem~\ref{th:shift_func_XS1_L}).
That result will be applied to diffeomorphisms preserving a Morse function on an orientable surface and reversing certain regular components of some of its level-sets (Theorems~\ref{th:sq_Sminus_Sid}, \ref{th:exist:2disk}, and~\ref{th:exist:g_rev_or}).

\section{Preliminaries}



\subsection{Action of the groups of diffeomorphisms}
Let $\Mman$ be a compact surface and $\Pman$ be either the real line $\bR$ or the circle $\Circle$ and $\Diff(\Mman)$ be the group of $\Cinfty$ diffeomorphisms of $\Mman$.
Then there is a natural right action of $\Diff(\Mman)$ on the space of smooth maps $\Ci{\Mman}{\Pman}$ defined by the following rule: $(\dif,\func)\mapsto \func\circ\dif$, where $\dif\in\Diff(\Mman)$, $\func\in \Ci{\Mman}{\Pman}$.
Let
\begin{align*}
  \Stabilizer{\func} &= \{\dif\in \Diff(\Mman) \mid \func\circ\dif = \func\}, &
  \Orbit{\func} &= \{ \func\circ\dif  \mid \dif\in\Diff(\Mman)\},
\end{align*}
be the \myemph{stabilizer} and the \myemph{orbit} of $\func\in \Ci{\Mman}{\Pman}$ with respect to the above action.
It will be convenient to say that elements of $\Stabilizer{\func}$ \myemph{preserve} $\func$.
Endow the spaces $\Diff(\Mman)$, $\Ci{\Mman}{\Pman}$ with Whitney $\Cinfty$-topologies, and their subspaces $\Stabilizer{\func}$, $\Orbit{\func}$ with induced ones.
Denote by $\StabilizerId{\func}$ the identity path component of $\Stabilizer{\func}$ consisting of all $\dif\in\Stabilizer{\func}$ isotopic to $\id_{\Mman}$ by some $\func$-preserving isotopy.

For orientable $\Mman$ we denote by $\Diff^{+}(\Mman)$ the group of its orientation preserving diffeomorphisms.
Also for $\func\in\Ci{\Mman}{\Pman}$ we put
\begin{align*}
\StabilizerPlus{\func} &= \Diff^{+}(\Mman) \cap \Stabilizer{\func}, &
\StabilizerMinus{\func} &= \Stabilizer{\func} \setminus \StabilizerPlus{\func},
\end{align*}
so $\StabilizerMinus{\func}$ is another adjacent class of $\Stabilizer{\func}$ by $\StabilizerPlus{\func}$ distinct from $\StabilizerPlus{\func}$.

\subsection{Homogeneous polynomials on $\bR^2$ without multiple factors}
It is well known and is easy to prove that every homogeneous polynomial $\func:\bR^2\to\bR$ is a product of finitely many linear and irreducible over $\bR$ quadratic factors: $\func = \prod\limits_{i=1}^{l}L_i \cdot \prod\limits_{j=1}^{q}Q_j$, where $L_i(\px,\py) = a_i \px + b_i \py$ and $Q_j(\px,\py) = c_i \px^2 + d_i \px\py + e_i \py^2$.

Suppose $\deg\func\geq2$.
Then the origin $0$ is a unique critical point of $\func$ if and only if $\func$ has no multiple (non-proportional each other) linear factors.
Structure of level sets of $\func$ near $0$ is shown in Figure~\ref{fig:level_sets}.
\begin{figure}[htbp!]\footnotesize
\begin{tabular}{ccccccc}
\includegraphics[height=1.4cm]{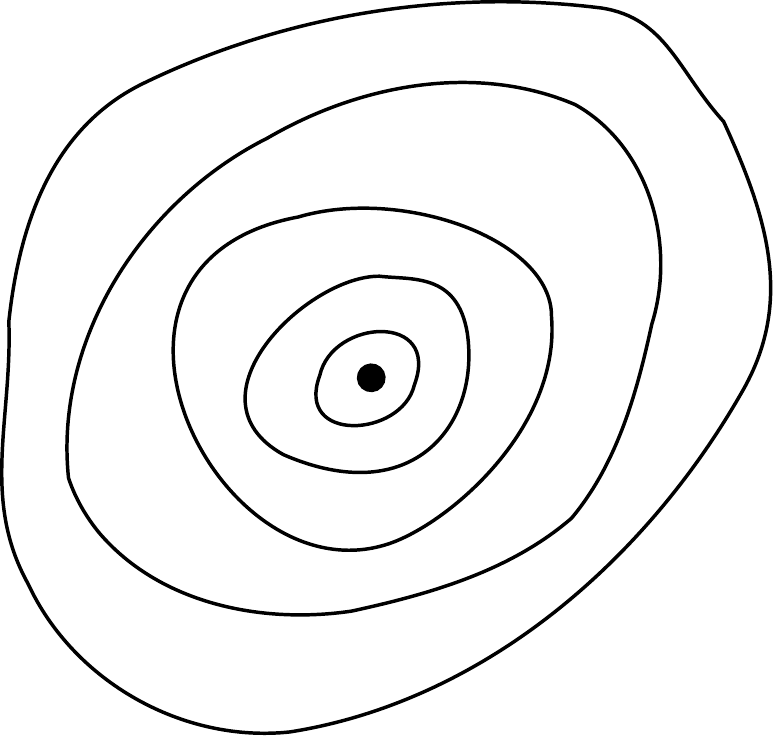}  & 
\includegraphics[height=1.4cm]{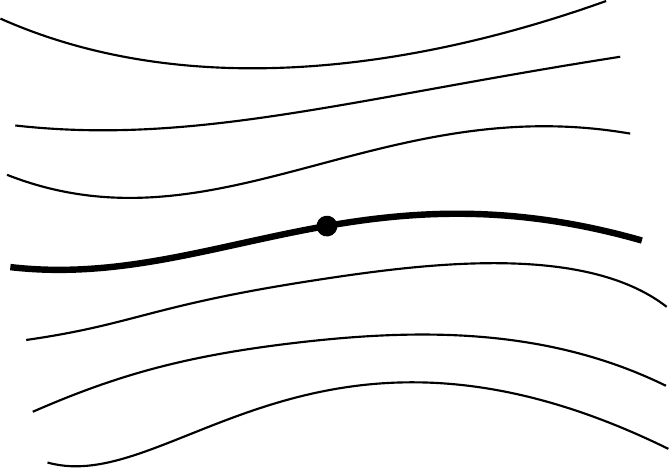}    & 
\includegraphics[height=1.4cm]{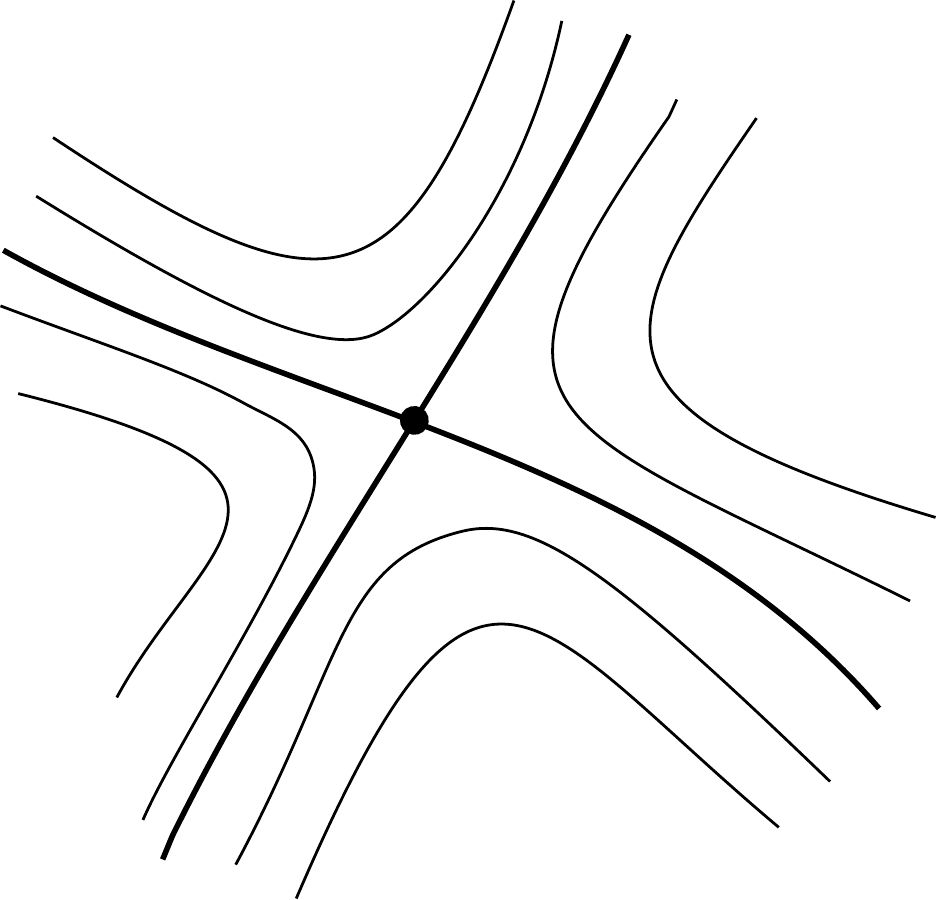}    & 
\includegraphics[height=1.4cm]{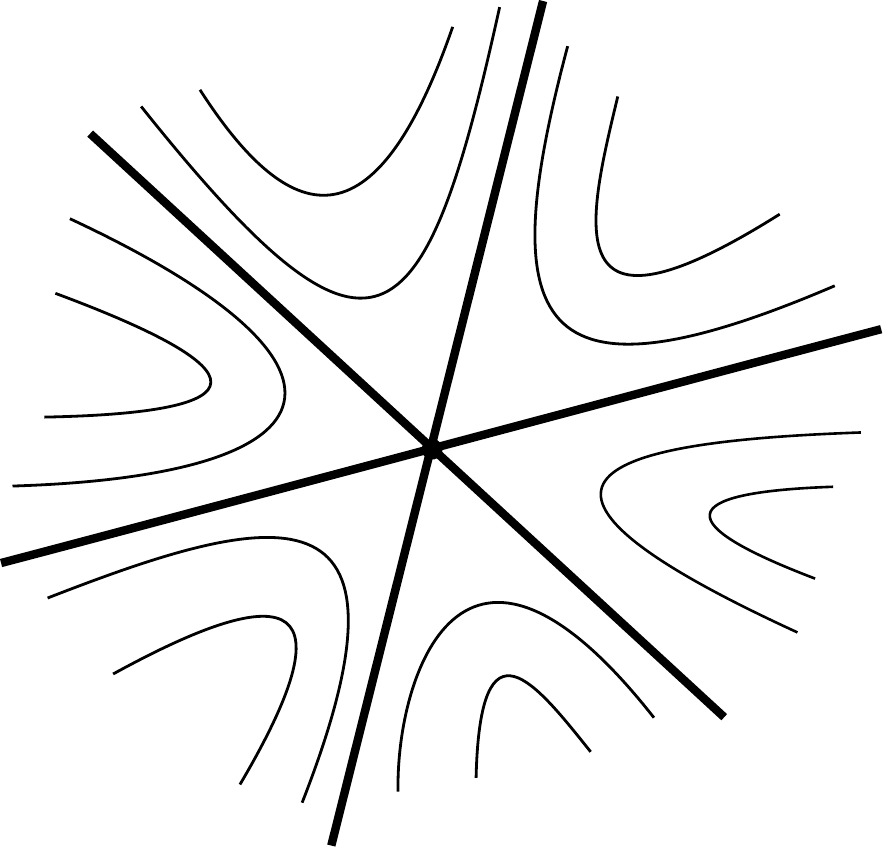}    \\
a) local extreme  &
b) quasi-saddle   &
c) non-degenerate &
d) saddle          \\
        &
        &
 saddle &
 $f = L_1 \cdots L_l Q_1\cdots Q_q$ \\
$f = Q_1\cdots Q_q$ &
$f = L_1 Q_1\cdots Q_q$ &
$f = L_1 L_2$ &
$l\geq 2$ \\
$l=0$, $q\geq 1$ &
$l=1$, $q\geq 1$ &
$l=2$, $q=0$ &

\end{tabular}
\caption{Topological structure of level sets of a homogeneous polynomial $\func:\bR^2\to\bR$ without multiple factors}
\label{fig:level_sets}
\end{figure}

We will restrict ourselves with the following space of smooth maps.

\begin{subdefinition}
Let $\FMP$ be the set of $\Cinfty$ maps $f:\Mman\to\Pman$ satisfying the following conditions:
\begin{enumerate}[leftmargin=*, label={\rm(A\arabic*)}]
\item\label{enum:F:bd}
The map $\func$ takes constant value at each connected component of $\partial\Mman$ and has no critical points in $\partial\Mman$.
\item\label{enum:F:hompol}
For every critical point $z$ of $f$ there is a local presentation $f_z\colon \bR^2 \to \bR$ of $\func$ near $z$ such that $f_z$ is a homogeneous polynomial $\bR^2 \to \bR$ without multiple factors.
\end{enumerate}
\end{subdefinition}

A map $\func\in \Ci{\Mman}{\Pman}$ will be called \myemph{Morse} if $\func$ satisfies~\ref{enum:F:bd} and all its critical points are non-degenerate.
Notice that due to Morse lemma each Morse map $\func$ satisfies the condition~\ref{enum:F:hompol} with homogeneous polynomials $f_z=\pm x^2\pm y^2$ for each critical point $z$.
Denote by $\Morse(\Mman,\Pman)$ the set of all Morse maps $\Mman\to\Pman$.
Then we have the following inclusion:
\[ \Morse(\Mman,\Pman)\subset \FMP \subset \Ci{\Mman}{\Pman}. \]
Then $\Morse(\Mman,\Pman)$ is open and everywhere dense in the subset of smooth functions from $\Ci{\Mman}{\Pman}$ satisfying the condition~\ref{enum:F:hompol}.
Hence $\FMP$ consists of \myemph{even more that typical maps $\Mman\to\Pman$}.

\subsection{Several constructions associated with $\func\in\FMP$.}
In what follows we will assume that $\func\in \FMP$.
Let $\fSing$ be the set of critical points of $f$.
Then condition~\ref{enum:F:hompol} implies that each $\pz\in\fSing$ is isolated.
In the case when $\Pman=\Circle$ one can also say about local extremes of $\func$, and even about local minimums or maximums if we fix an orientation of $\Circle$.

A connected component $\Kman$ of a level-set $f^{-1}(c)$, $c\in\Pman$, will be called a \myemph{leaf} (of $\func$).
We will call $\Kman$ \myemph{regular} if it contains no critical points.
Otherwise, it will be called \myemph{critical}.

For $\eps>0$ let $\Nman_{\eps}$ be the connected component of $f^{-1}[c-\varepsilon, c+\varepsilon]$ containing $\Kman$.
Then $\Nman_{\eps}$ will be called an \myemph{$\func$-regular neighborhood of $\Kman$} if $\eps$ is so small that $\Nman_{\eps}\setminus\Kman$ contains no critical points of $\func$   and no boundary components of $\partial\Mman$.

A submanifold $\Uman\subset\Mman$ will be called \myemph{$\func$-adapted} if $\Uman = \mathop{\cup}\limits_{i=1}^{a} \Aman_i$, where each $\Aman_i$ is either a regular leaf of $\func$ or an $\func$-regular neighborhood of some (regular or critical) leaf of $\func$.

\subsubsection{Graph of $\func$}
Let $\KRGraphf$ be the partition of $\Mman$ into leaves of $\func$, and $\prj:\Mman\to\KRGraphf$ be the natural map associating to each $x\in\Mman$ the corresponding leaf of $\func$ containing $\px$.
Endow $\KRGraphf$ with the quotient topology, so a subset (a collection of leaves) $\Uman\subset \KRGraphf$ is open iff $p^{-1}(\Uman)$ (i.e. their union) is open in $\Mman$.
It follows from axioms~\ref{enum:F:bd} and~\ref{enum:F:hompol} that $\KRGraphf$ has a natural structure of $1$-dimensional CW-complex, whose $0$-cells correspond to boundary components of $\Mman$ and critical leaves of $\func$.
It is known as Reeb or Kronrod-Reeb or Lyapunov graph of $\func$, see~\cite{Adelson-Welsky-Kronrode:DAN:1945, Reeb:CR:1946, Kronrod:UMN:1950, Franks:Top:1985}.
We will call $\KRGraphf$ simply the \myemph{graph} of $\func$.

The following statement is evident:
\begin{sublemma}\label{lm:KR_is_a_tree}
For $\func\in\FMP$ the following conditions are equivalent:
\begin{enumerate}[label={\rm(\alph*)}]
\item\label{enum:lm:KR_is_a_tree:1} every regular leaf of $\func$ in $\Int{\Mman}$ separates $\Mman$;
\item\label{enum:lm:KR_is_a_tree:2} the graph $\KRGraphf$ of $\func$ is a tree.
\end{enumerate}
For instance, those conditions hold if $\Mman$ is either of the surfaces: $2$-disk, cylinder, $2$-sphere, M\"obius band, projective plane.
\qed
\end{sublemma}


\subsubsection{Singular foliation of $\func$}
Consider finer partition $\SingFol$ of $\Mman$ whose elements are of the following three types:
\begin{enumerate}[label={\rm(\roman*)}]
\item regular leaves of $\func$;
\item critical points of $\func$;
\item connected components of the sets $\crLev\setminus\fSing$, where $\crLev$ runs over all critical leaves of $\func$ (evidently, each such component is an open arc).
\end{enumerate}
In other words, each critilal leaf of $\func$ is additionally partitioned by critical points.
We will call $\SingFol$ the \myemph{singluar foliation of $\func$}.

\subsubsection{Hamiltonian like flows of $\func\in\FMP$}
Suppose $\Mman$ is orientable.
A smooth vector field $\fld$ on $\Mman$ will be called \myemph{Hamiltonian-like} for $\func$ if the following conditions hold true.
{\em\begin{enumerate}[leftmargin=*, label={\rm(\alph*)}, itemsep=1ex]
\item\label{enum:HamVF:F_0_crpt}
$\fld(z)=0$ if and only if $z$ is a critical point of $\func$.
\item\label{enum:HamVF:Ff_0}
$d\func(\fld)\equiv0$ everywhere on $\Mman$, i.e. $\func$ is constant along orbits of $\fld$.
\item\label{enum:HamVF:local_form}
Let $z$ be a critical point of $\func$.
Then there exists a local representation of $\func$ at $z$ in the form of a homogeneous polynomial $\gfunc:\bR^2\to\bR$ without multiple factors as in Axiom~\ref{enum:F:hompol}, such that in the same coordinates $(x,y)$ near the origin $0$ in $\bR^2$ we have that $\fld = -\gfunc'_{y}\,\tfrac{\partial}{\partial x} + \gfunc'_{x}\,\tfrac{\partial}{\partial y}$.
\end{enumerate}}

It follows from~\ref{enum:HamVF:F_0_crpt} and Axiom~\ref{enum:F:bd} that the orbits of $\fld$ precisely elements of the singular foliation $\SingFol$.

By~\cite[Lemma~5.1]{Maksymenko:AGAG:2006} or~\cite[Lemma~16]{Maksymenko:ProcIM:ENG:2010} every $\func\in\FF(\Mman,\Pman)$ admits a Hamiltonian-like vector field.

The folowing statement is a principal technical result established in a series of papers by the second author, see~\cite[Lemma~6.1(iv)]{Maksymenko:TA:2020} for details:
\begin{sublemma}\label{lm:char_Sid}
Let $\flow:\Mman\times\bR\to\Mman$ be a Hamiltonian like flow for $\func$ and $\dif\in\Stabilizer{\func}$.
Then $\dif\in\StabilizerId{\func}$ if and only if there exists a $\Cinfty$ function $\alpha:\Mman\to\bR$ such that $\dif(\px) = \flow(\px,\alpha(\px))$ for all $\px\in\Mman$.
Moreover, in this case the homotopy $H:\Mman\times[0;1] \to \Mman$ given by $H(\px,\pt) = \flow(\px,\pt\alpha(\px))$ is an isotopy between $H_0 = \id_{\Mman}$ and $H_1 =\dif$ in $\Stabilizer{\func}$.
\end{sublemma}

\section{Main results}
Let $\Mman$ be a compact surface, $\dif:\Mman\to\Mman$ a homeomorphism, and $\gamma \subset \Mman$ a submanifold of $\Mman$.
Then $\gamma$ is \myemph{$\dif$-invariant}, whenever $\dif(\gamma)=\gamma$.

Moreover, suppose $\gamma$ is connected and \myemph{orientable} $\dif$-invariant submanifold with $\dim\gamma\geq1$.
Then $\gamma$ is \myemph{$\pinv{\dif}$-invariant} (resp. \myemph{$\minv{\dif}$-invariant}) whenever the restriction $\dif|_{\gamma}:\gamma\to\gamma$ preserves (resp. reverses) orientation of $\gamma$.
If $\gamma$ is a fixed point of $\dif$, then we assume that $\gamma$ is mutually $\pinv{\dif}$- and $\minv{\dif}$-invariant.

We will be interesting in the structure of diffeomorphisms preserving $\func\in\FMP$ and reversing orientations of some regular leaves of $\func$.
The following easy lemma can be proved similarly to~\cite[Lemma~3.5]{Maksymenko:AGAG:2006}.
\begin{lemma}\label{lm:pres_fol}
Let $\Mman$ be a connected orientable surface and $\dif\in\Stabilizer{\func}$ be such that every regular leaf of $\func$ is $\dif$-invariant.
Then every critical leaf of $\func$ in also $\dif$-invariant and the following conditions are equivalent:
\begin{enumerate}[label={\rm(\arabic*)}]
\item some regular leaf of $\func$ is $\pinv{\dif}$-invariant;
\item all regular leaves of $\func$ are $\pinv{\dif}$-invariant;
\item $\dif$ preserves orientation of $\Mman$.
\end{enumerate}
\end{lemma}

\begin{counterexample}
Lemma~\ref{lm:pres_fol} fails for non-orientable surfaces.
Let $\Mman$ be a M\"obius band, and $\func:\Mman\to\bR$ be a Morse function having two critical points: one local extreme $\pz$ and one saddle $\py$.
Let $\Kman$ be the critical leaf of $\func$ containing $\py$, and $\Dman$ and $\Eman$ be the connected component of $\Mman\setminus\Kman$ containing $\pz$ and $\partial\Mman$ respectively.
Then it is easy to construct a diffeomorphism $\dif:\Mman\to\Mman$ (called \myemph{slice along central circle of M\"obius band}) which is fixed on $\partial\Mman$, and for which
$\Dman$ is $\minv{\dif}$-invariant.
Then regular leaves of $\func$ in $\Dman$ are $\minv{\dif}$-invariant, while regular leaves of $\func$ in $\Eman$ are $\pinv{\dif}$-invariant.
\end{counterexample}

Denote by $\MFolStabilizer{\func}$ the \myemph{subset} of $\Stabilizer{\func}$ consisting of diffeomorphisms $\dif$ such that \myemph{every regular leaf of $\func$ is $\minv{\dif}$-invariant}.

Let $\dif\in\MFolStabilizer{\func}$.
Then by Lemma~\ref{lm:pres_fol} every critical leaf $\Kman$ of $\func$ is $\minv{\dif}$-invariant, however, $\dif$ may interchange critical points of $\func$ in $\Kman$ and the leaves of $\SingFol$ contained in $\Kman$ (i.e.~connected components of $\Kman\setminus\fSing$).
Notice also that in general $\MFolStabilizer{\func}$ can be empty.

The following theorem can be regarded as a \myemph{homotopical} and \myemph{foliated} variant of the above rigidity property for diffeomorphisms preserving $\func$.

\begin{theorem}\label{th:sq_Sminus_Sid}
Let $\Mman$ be a connected compact orientable surface, $\func\in\FMP$, and $\dif\in \MFolStabilizer{\func}$.
Then $\dif^2\in\StabilizerId{\func}$.
In particular, for every $\dif\in\MFolStabilizer{\func}$ each leaf of $\SingFol$ is $(\dif^2,+)$-invariant.
\end{theorem}
\begin{corollary}
Let $\ori:\pi_0 \Stabilizer{\func} \to \bZ_2 = \{0,1\}$ be the orientation homomorphism defined by $\ori(\StabilizerPlus{\func})=0$ and  $\ori(\StabilizerMinus{\func})=1$.
Then for each $\dif\in\MFolStabilizer{\func}$ the map $s:\bZ_2\to \pi_0\Stabilizer{\func}$, defined by $s(0)=[\id_{\Mman}]$ and $s(1)=[\dif]$ is a homomorphism satisfying $\ori\circ s = \id_{\bZ_2}$.
In other words, $s$ is a section of $\ori$, whence $\pi_0\Stabilizer{\func}$ is a certain semidirect product of $\pi_0\StabilizerPlus{\func}$ and $\bZ_2$.
\end{corollary}
\begin{proof}
One should just mention that Theorem~\ref{th:sq_Sminus_Sid} implies that $[\dif]^2 = [\id_{\Mman}]$ in $\pi_0\Stabilizer{\func}$.
\end{proof}

Now consider the situation, when not all regular leaves are $\minv{\dif}$-invariant.
To clarify the situation we first formulate a particular case of the general statement for maps on $2$-disk and cylinder.

\begin{theorem}\label{th:exist:2disk}
Let $\Mman$ be a connected compact orientable surface, $\func\in\FMP$, $\Vman$ be a regular leaf of $\func$, and $\dif\in\Stabilizer{\func}$.
Suppose every regular leaf of $\func$ in $\Int{\Mman}$ separates $\Mman$ (this holds e.g. when $\Mman$ is a $2$-disk, cylinder or a $2$-sphere, see Lemma~\ref{lm:KR_is_a_tree}), and that $\Vman$ is $\minv{\dif}$-invariant.
Then there exists $\gdif\in\Stabilizer{\func}$ which coincide with $\dif$ on some neighborhood of $\Vman$ and such that $\gdif^2 \in \StabilizerId{\func}$.
\end{theorem}

Emphasize that Theorem~\ref{th:exist:2disk} \myemph{does not claim that $\gdif\in\MFolStabilizer{\func}$}, though $\gdif$ reverses orientation of $\Vman$ and its square belong to $\StabilizerId{\func}$.

\subsection*{General result}
Theorems~\ref{th:sq_Sminus_Sid} and~\ref{th:exist:2disk} are consequences of the following Theorem~\ref{th:exist:g_rev_or} below.
Let $\Mman$ be a connected compact (not necessarily orientable) surface, $\func\in\FMP$, and $\dif\in\Stabilizer{\func}$.
Let also
\begin{enumerate}[leftmargin=6ex, label={$\bullet$}, itemsep=0.8ex]
\item
$\Aman$ be the union of \myemph{all $\minv{\dif}$-invariant regular leaves of $\func$},
\item
$\Kman_1,\ldots,\Kman_{\vk}$ be all the critical leaves of $\func$ such that $\overline{\Aman}\cap\Kman_i\not=\varnothing$;
\item
for $i=1,\ldots,\vk$, let $\regN{\Kman_i}$ be an $\func$-regular neighborhood of $\Kman_i$ chosen so that $\regN{\Kman_i} \cap \regN{\Kman_j} = \varnothing$ for $i\not= j$ and
\[
   \Zman  :=  \Aman \,\bigcup\,  \Bigl( \mathop{\cup}\limits_{i=1}^{\vk} \regN{\Kman_i} \Bigr).
\]
Evidently, $\Zman$ is an $\func$-adapted subsurface of $\Mman$ and each of its connected components intersects $\Aman$.
\end{enumerate}

\begin{sublemma}\label{lm:dZ_prop}
Suppose $\Zman$ is orientable.
Let also $\gamma$ be a boundary component of $\Zman$.
Consider the following conditions:
\begin{enumerate}[label={\rm(\arabic*)}]
\item\label{enum:lm:dZ_prop:1} $\gamma \subset \Int{\Mman}$;
\item\label{enum:lm:dZ_prop:2} $\gamma = \partial\Uman\cap\partial\Zman$ for some connected component $\Uman$ of $\overline{\Mman\setminus\Zman}$;
\item\label{enum:lm:dZ_prop:3} $\dif(\gamma) \cap \gamma = \varnothing$;
\item\label{enum:lm:dZ_prop:4} $\dif(\gamma) \not= \gamma$.
\end{enumerate}
Then \ref{enum:lm:dZ_prop:1}$\Leftrightarrow$\ref{enum:lm:dZ_prop:2} $\Rightarrow$ \ref{enum:lm:dZ_prop:3}$\Leftrightarrow$\ref{enum:lm:dZ_prop:4}.
\end{sublemma}
\begin{proof}
\ref{enum:lm:dZ_prop:1}$\Leftrightarrow$\ref{enum:lm:dZ_prop:2} is evident, and~\ref{enum:lm:dZ_prop:3}$\Leftrightarrow$\ref{enum:lm:dZ_prop:4} follows from the observation that $\gamma$ is a regular leaf of $\dif$ and $\dif$ permutes leaves of $\func$.

\ref{enum:lm:dZ_prop:2}$\Rightarrow$\ref{enum:lm:dZ_prop:4}
Suppose $\dif(\gamma) = \gamma$.
Let $\Zman'$ be a connected component of $\Zman$ containing $\gamma$.
Then by the construction $\Zman'$ must intersect $\Aman$, whence $\dif$ changes orientation of some regular leaves in $\Zman'$.
Since $\Zman'$ is also orientable, we get from Lemma~\ref{lm:pres_fol} that $\dif$ also changes orientation of $\gamma$.
This means that $\gamma\subset\Aman$, and therefore there exists an open neighborhood $\Wman \subset \Aman$ of $\gamma$ consisting of regular leaves of $\func$.
In particular, the regular leaves in $\Wman\cap \Int{\Uman}$ must be contained in $\Aman$ which contradicts to the assumption that $\Int{\Uman} \subset \Mman\setminus \Zman \subset  \Mman\setminus \Aman$.
\end{proof}

Thus $\dif$ interchanges boundary components of $\Zman$ belonging to the interior of $\Mman$.
We will introduce the following property on $\Zman$:
\begin{enumerate}[label={\rm(B)}]
\item\label{cond:B}
{\it every connected component of $\partial\Zman \cap \Int{\Mman}$ separates $\Mman$.}
\end{enumerate}

This condition means that there is a bijection between boundary components of $\partial\Zman \cap \Int{\Mman}$ and connected components of $\Mman\setminus\Zman$, whence by Lemma~\ref{lm:dZ_prop} there will be no $\dif$-invariant connected components of $\Mman\setminus\Zman$.

\begin{theorem}\label{th:exist:g_rev_or}
If $\Zman$ is non-empty, orientable and has property~\ref{cond:B}, then there exists $\gdif\in\Stabilizer{\func}$ such that $\gdif=\dif$ on $\Zman$ and $\gdif^2 \in \StabilizerId{\func}$.
\end{theorem}

\subsection*{Proof of Theorem~\ref{th:sq_Sminus_Sid}.}
Suppose $\Mman$ is orientable and $\dif\in\MFolStabilizer{\func}$.
Then in the notation of Theorem~\ref{th:exist:g_rev_or}, $\Zman = \Mman$, and by that theorem there exists $\gdif\in\Stabilizer{\func}$ such that $\gdif=\dif$ on $\Zman$ and $\gdif^2\in\StabilizerId{\func}$.
This means that $\dif=\gdif$ and $\dif^2\in\StabilizerId{\func}$.
\qed

\subsection*{Proof of Theorem~\ref{th:exist:2disk}.}
The assumption that every regular leaf of $\func$ in $\Int{\Mman}$ separates $\Mman$ implies condition~\ref{cond:B}.
\qed



\subsection{Structure of the paper}
In Section~\ref{sect:shift} we discuss a notion of a shift map along orbits of a flow which was studied in a series of papers by the second author, and extend several results to continuous flows.
Section~\ref{sect:circle_maps} devoted to reversing orientation families of homeomorphisms of the circle.
In Section~\ref{sect:flows_no_fixpt} we study flows without fixed point, and in Section~\ref{sect:polar_coord} recall several results about passing from a flow on the plane to the flow written in polar coordinates.
Section~\ref{sect:chip_cyl} introduces a certain subsurfaces of a surface $\Mman$ associated with a map $\func\in\FMP$ and called \myemph{chipped cylinders}.
We prove Theorem~\ref{th:shift_func_on_chip_nbh} describing behaviour of diffeomorphisms reversing regular leaves of $\func$ contained in those chipped cylinders.
In Section~\ref{sect:change_dif} we prove Lemma~\ref{lm:gm_in_Sid} allowing to change $\func$-preserving diffeomorphisms so that its finite power will be isotopic to the identity by $\func$-preserving isotopy.
Finally, in sections~\ref{sect:change_dif} and~\ref{sect:proof:th:exist:g_rev_or} we prove Theorem~\ref{th:exist:g_rev_or}.

\section{Shifts along orbits of flows}\label{sect:shift}
In this section we extend several results obtained in~\cite{Maksymenko:TA:2003, Maksymenko:TA:2020} for smooth flows to a continuous situation.
Let $\Xman$ be a topological space.

\begin{definition}
A continuous map $\flow:\Xman\times\bR\to\Xman$ is a \myemph{(global) flow on $\Xman$}, if $\flow_0 = \id_{\Xman}$ and $\flow_{\ps}\circ\flow_{\pt} =\flow_{\ps+\pt}$ for all $\ps,\pt\in\bR$, where $\flow_{\ps}: \Xman\to\Xman$ is given by $\flow_{\ps}(\px) = \flow(\px,\ps)$.
For $\px\in\Xman$ the subset $\flow(\px\times\bR)\subset \Xman$ is called the \myemph{orbit} of $\px$.
\end{definition}

\smallskip

It is well known that a $\Cr{r}$, $1\leq r\leq \infty$, vector field $\fld$ of a smooth compact manifold $\Xman$ tangent to $\partial\Xman$ always generates a flow $\flow$.

Assume further that $\flow$ is a flow on a topological space $\Xman$.
Let also $\Vman \subset \Xman$ be a subset.
Say that a continuous map $\dif:\Vman\to\Xman$ \myemph{preserves orbits of $\flow$ on $\Vman$} if $\dif(\gamma\cap\Vman) \subset \gamma$ for every orbit $\gamma$ of $\flow$.
The latter means that for each $x\in\Uman$ there exists a number $\alpha_{x}\in\bR$ such that $\dif(x) = \flow(x,\alpha_{x})$.
Notice that in general $\alpha_{x}$ is not unique and does not continuously depend on $x$.

Conversely, let $\alpha:\Vman\to \bR$ be a continuous function such that its graph $\Gamma_{\alpha}=\{(\px,\alpha(\px) \mid \px\in\Vman\}$ is contained in $\Wman$.
For a global flow any continuous $\alpha:\Vman\to \bR$ satisfies that condition.
Then one can define the following map $\flow_{\alpha}:\Vman\to \Xman$ by
\[
   \flow_{\alpha}(\pz) := \flow(\pz,\alpha(\pz)) = \flow_{\alpha(\pz)}(\pz).
\]
We will call $\flow_{\alpha}$ a \myemph{shift along orbits of $\flow$ by the function $\alpha$}, while $\alpha$ will be called a \myemph{shift function} for $\flow_{\alpha}$.

Notice that $\flow_{\alpha}$ preserves orbits of $\flow$ on $\Vman$, and in general is not a homeomorphism.

\begin{definition}\label{def:tub_nbh_prop}
Let $\px$ be a non-fixed point of $\flow$, and $\Yman$ be a topological space.
Let also $\Uman$ be an open neighborhood of $\px$, and $\phi=(\zeta,\rho):\Uman\to \Yman\times\bR$ be an open embedding.
Then the pair $(\phi,\Uman)$ will be called a \myemph{flow-box chart} at $\px$, if there exist an $\eps>0$ and an open neighborhood $\Vman$ of $\px$ in $\Xman$ such that
\begin{equation}\label{equ:flow_box}
\phi \circ \flow(\pz,\pt) = \bigl( \zeta(\pz), \rho(\pz) + \pt  \bigr)
\end{equation}
for all $(\pz,\pt) \in \Vman\times(-\eps;\eps)$.

In other words, $\flow$ is \myemph{locally conjugated} to the flow $\gflow:(\Yman\times\bR)\times\bR \to \Yman\times\bR$ defined by $\gflow(\py,\ps,\pt) = (\py,\ps+\pt)$, since the identity~\eqref{equ:flow_box} can be written as
\[
  \phi \circ \flow(\pz,\pt) = \bigl( \zeta(\pz), \rho(\pz) + \pt  \bigr) =
  \gflow\bigl( \zeta(\pz), \rho(\pz), \pt \bigr) =
  \gflow( \phi(\pz), \pt),
\]
i.e.
\begin{equation}\label{equ:phi_F__G_phi}
  \phi\circ\flow_{\pt}(\pz) = \gflow_{\pt} \circ \phi(\pz).
\end{equation}
\end{definition}

It is well known that each $\Cr{2}$ flow (generated by some $\Cr{1}$ vector field) on a smooth manifold admits flow-box charts at each non-fixed point.

\begin{lemma}
If $(\phi,\Uman)$ is a flow box chart at a non-fixed point $\px\in\Xman$ of $\flow$, then for any $\tau\in\bR$ the pair $(\phi\circ \flow_{-\tau}, \flow_{\tau}(\Uman))$ is a flow box chart at the point $\py = \flow(\px,\tau)$.
\end{lemma}
\begin{proof}
Denote $\Uman' = \flow_{\tau}(\Uman)$ and $\phi' = \phi\circ \flow_{-\tau}$.
Let also $\eps$ and $\Vman$ be as in Definition~\ref{equ:flow_box}, and $\Vman' =  \flow_{\tau}(\Vman)$ be a neighborhood of $\py$.
Then for any $(\pz,\pt) \in \Vman' \times (-\eps;\eps)$ we have that
\begin{align*}
  \phi'  \circ \flow_{\pt}(\pz)
    &= (\phi\circ \flow_{-\tau}) \circ \flow_{\pt}(\pz)
     = (\phi\circ \flow_{\pt-\tau})(\pz) =  \\
    &= (\phi\circ \flow_{\pt} \bigl( \flow_{-\tau}(\pz) \bigr)
     \stackrel{\eqref{equ:flow_box}}{=\!=} \gflow_{\pt} \circ \phi(\flow_{-\tau}(\pz)) = \gflow_{\pt} \circ \phi'(\pz).
    \qedhere
  \end{align*}
\end{proof}

The following lemma shows that for flows admitting flow box charts (e.g. for smooth flows) every orbit preserving map admits a shift function near each non-fixed point.
Moreover, such a function is locally determined by its value at that point.

\begin{lemma}\label{lm:local_uniqueness_shift_func}{\rm(cf.~\cite[Lemma~6.1(i)]{Maksymenko:TA:2020}).}
Suppose a flow $\flow:\Xman\times\bR\to\Xman$ has a flow-box $(\phi,\Uman)$ at some non-fixed point $\px$.
Let also $\dif:\Uman\to\Xman$ be a continuous map preserving orbits of $\flow$ and such that $\dif(\px) = \flow(\px,\tau)$ for some $\tau\in\bR$.
Then there exists an open neighborhood $\Vman\subset\Uman$ of $\px$ and a unique continuous function $\alpha:\Vman\to\bR$ such that
\begin{enumerate}[label={\rm(\arabic*)}]
\item
$\alpha(\px) = \tau$;
\item
$\dif(\pz) = \flow(\pz,\alpha(\pz))$ for all $\pz\in\Vman$.
\end{enumerate}
If in addition $\Xman$ is a manifold of class $\Cr{r}$, $(1\leq r \leq \infty)$, $\flow$ and $\dif$ are $\Cr{r}$, and $\phi$ is a $\Cr{r}$ embedding, then $\alpha$ is $\Cr{r}$ as well.
\end{lemma}
\begin{proof}
The proof almost literally repeats the arguments of~\cite[Lemma~6.2]{Maksymenko:TA:2020} proved for smooth flows and based on existence of flow box charts.
For completeness we present a short proof for continuous situation.

1) First suppose $\tau=0$, so $\dif(\px) = \px$.
Let $\Vman$ and $\eps$ be the same as in Definition~\ref{def:tub_nbh_prop}.
Decreasing $\Vman$ one can also assume that $\Vman \subset \Uman \cap \dif^{-1}(\Uman)$, so in particular $\dif(\Vman) \subset \Uman$.
Denote $\hat{\Vman} := \phi(\Vman)$ and $\hat{\Uman} := \phi(\Uman)$.
Then these sets are open, and we have a well-defined map $\hat{\dif} = \phi \circ\dif\circ\phi^{-1}: \hat{\Vman} \to \hat{\Uman}$ which preserves orbits of $\gflow$ due to~\eqref{equ:phi_F__G_phi}.
This means that $\hat{\dif}(\py,\pt) = \bigl(\py, \eta(\py,\pt) \bigr)$ for some continuous function $\Yman\times\bR\supset\hat{\Vman} \xrightarrow{~\eta~} \bR$.
Define another continuous function $\alpha':\hat{\Vman} \to \bR$ by $\alpha'(\py, \pt) =  \eta(\py,\pt)  - \pt$.
Then
\[
  \hat{\dif}(\py,\pt)
   = \bigl(\py, \pt +\alpha'(\py, \pt)  \bigr)
   = \gflow\bigl(\py,\pt, \alpha'(\py, \pt)\bigr)
   = \gflow_{\alpha'}(\py,\pt),
\]
that is $\phi \circ\dif\circ\phi^{-1} = \hat{\dif} = \gflow_{\alpha'}$, whence for each $\pz\in\Vman$ we have that
\[
  \dif(\pz) = \phi^{-1} \circ \gflow_{\alpha'} \circ\phi(\pz) =
  \phi^{-1} \circ \gflow_{\alpha'\circ\phi(\pz)} \circ \phi(\pz) = \flow_{\alpha'\circ\phi(\pz)}(\pz) =
  \flow_{\alpha'\circ\phi}(\pz).
\]
Thus one can put $\alpha = \alpha'\circ\phi:\Vman\to\bR$.
It follows that $\alpha$ is continuous.
Moreover, let $\hat{\py},\hat{\pt}) = \phi(\px)$.
Since $\dif(\px)=\px$, we get that $\hdif(\hat{\py},\hat{\pt}) = \hat{\py},\hat{\pt})$, whence $\eta(\hat{\py},\hat{\pt})=\hat{\pt}$, whence $\alpha'(\hat{\py},\hat{\pt}) = 0$, and thus
\[ \alpha(\px) = \alpha'\circ\phi(\px) = \alpha'(\hat{\py},\hat{\pt}) = 0. \]

2) If $\tau\not=0$, then consider the map $\hdif:\Uman\to\Xman$ given by $\hdif = \flow_{-\tau} \circ \dif$.
Then $\hdif(\px)= \flow_{-\tau} \circ \dif(\px) = \flow( \flow(\px, \tau), -\tau) = \px$, whence by 1) $\hdif = \flow_{\hat{\alpha}}$ for a unique continuous function $\hat{\alpha}:\Vman\to\bR$ such that $\hat{\alpha}(\px)=0$.
Put $\alpha = \hat{\alpha} + \tau$.
Then $\alpha(\px) = \tau$ and
\[
  \dif(\pz) = \flow_{\tau} \circ\hdif(\pz) = \flow( \flow(\pz, \hat{\alpha}(\pz)), \tau) =  \flow(\pz, \hat{\alpha}(\pz) + \tau)  =
  \flow(\pz, \alpha(\pz)).
  \qedhere
\]

The formulas for $\alpha$ imply that if $\Xman$, $\flow$, $\phi$, and $\dif$ are $\Cr{r}$, $(1\leq r\leq \infty)$, then $\alpha$ is $\Cr{r}$ as well.
\end{proof}

\begin{corollary}\label{cor:local_uniqueness}
Suppose $\Uman \subset \Xman$ is an open \myemph{connected} subset such that every $\px\in\Uman$ is non-fixed and admits a flow-box.
Let also $\alpha,\alpha':\Uman \to \bR$ be two continuous functions such that $\flow_{\alpha} = \flow_{\alpha'}$ on $\Uman$.
If $\alpha(\px) = \alpha'(\px)$ at some $\px\in\Uman$ (this holds e.g.~if $\flow$ has at least one non-periodic point in $\Uman$), then $\alpha=\alpha'$ on $\Uman$.
\end{corollary}
\begin{proof}
The set $A = \{ \alpha(\py) = \alpha'(\py) \mid \py\in \Uman \}$ is closed and by assumption non-empty (contains $\px$).
Moreover, by Lemma~\ref{lm:local_uniqueness_shift_func} this set is open, whence $A = \Uman$.
\end{proof}

\begin{corollary}\label{cor:shift_func_all_orb_nonclosed}
Let $\flow:\Xman\times\bR\to\Xman$ be a continuous flow, $\Uman\subset\Xman$ be an open subset such that every point $\px\in\Uman$ is non-fixed and non-periodic and admits a flow box chart, and $\dif:\Uman\to\Xman$ be an orbit preserving map.
Then there exists a unique continuous function $\alpha:\Uman\to\bR$ such that $\dif(\px) = \flow(\px,\alpha(\px))$ for all $\px\in\Uman$.

If in addition $\Xman$ is a manifold of class $\Cr{r}$, $(0\leq r \leq \infty)$, $\flow$ is $\Cr{r}$ and admits $\Cr{r}$ flow box charts, and $\dif$ is $\Cr{r}$, then $\alpha$ is $\Cr{r}$ as well.
\end{corollary}
\begin{proof}
Since every point $\px\in\Uman$ is non-fixed and non-periodic, there exists a unique number $\alpha(\px)$ such that $\dif(\px) = \flow(\px,\alpha(\px))$.
Moreover, since $\flow$ admits flow box chart at $\px$, it follows from Lemma~\ref{lm:local_uniqueness_shift_func} that the correspondence $\px\mapsto\alpha(\px)$ is a continuous function $\alpha:\Uman\to\bR$, which is also $\Cr{r}$ under the corresponding smoothness assumptions on $\Xman$, $\flow$, $\phi$, and $\dif$.
\end{proof}

The following statement extends some results established in~\cite{Maksymenko:TA:2020} for smooth flows to continuous flows having flow box charts at each non-fixed point on arbitrary topological spaces.
\begin{lemma}\label{lm:lift_of_shifts}{\rm(c.f.~\cite[Lemmas~6.1-6.3]{Maksymenko:TA:2020})}
Let $\flow:\Xman\times\bR\to\Xman$ be a flow, $\prj:\tXman\to\Xman$ a covering map, and $\xi:\tXman\to\tXman$ a covering transformation, i.e. a homeomorphism such that $\prj\circ\xi=\prj$.
Then the following statements hold.
\begin{enumerate}[label={\rm(\arabic*)}, wide]
\item\label{enum:lm:lift_of_shifts:lift}
$\flow$ lifts to a unique flow $\hflow:\tXman\times\bR\to\tXman$ such that $\prj \circ \hflow_{\pt} = \flow_{\pt}\circ \prj$ for all $\pt\in\bR$.

\item\label{enum:lm:lift_of_shifts:shift}
For each continuous function $\alpha:\Xman\to\bR$ the map $\hflow_{\alpha\circ\prj}:\tXman\to\bR$ is a lifting of $\flow_{\alpha}$, that is

\item\label{enum:lm:lift_of_shifts:comm}
$\hflow$ commutes with $\xi$ in the sense that $\hflow_{\pt}\circ\xi=\xi\circ\hflow_{\pt}$ for all $\pt\in\bR$.
More generally, for any function $\beta:\tXman\to\bR$ we have $\hflow_{\beta} \circ\xi = \xi \circ \hflow_{\beta\circ\xi}$.

\item\label{enum:lm:lift_of_shifts:implications}
For a continuous function $\beta:\tXman\to\bR$ and a point $\pz\in\tXman$ consider the following (``global'' and ``point'') conditions:

\begin{tabular}{p{5cm}p{6cm}}
\begin{enumerate}[label={\rm(g\arabic*)}, itemsep=0.4ex ]
  \item\label{enum:xx:ha__ha_xi}     $\beta = \beta \circ \xi$;
  \item\label{enum:xx:Fha_Fhaxi}     $\hflow_{\beta} = \hflow_{\beta\circ\xi}$;
  \item\label{enum:xx:Fhaxi_xiFha}   $\hflow_{\beta}\circ\xi = \xi\circ\hflow_{\beta}$;
\end{enumerate}
&
\begin{enumerate}[label={\rm(p\arabic*)}, itemsep=0.4ex ]
  \item\label{enum:xx:ha__ha_xi_z}   $\beta(\pz) = \beta \circ \xi(\pz)$;
  \item\label{enum:xx:Fha_Fhaxi_z}   $\hflow_{\beta}(\pz) = \hflow_{\beta\circ\xi}(\pz)$;
  \item\label{enum:xx:Fhaxi_xiFha_z} $\hflow_{\beta}\circ\xi(\pz) = \xi\circ\hflow_{\beta}(\pz)$.
\end{enumerate}
\end{tabular}

Then we have the following diagram of implications:
\begin{equation}\label{equ:diag_implications}
\xymatrix{
  \ref{enum:xx:ha__ha_xi}    \ar@{=>}[r]      \ar@{=>}[d] &
  \ref{enum:xx:Fha_Fhaxi}    \ar@{<=>}[r]^{\ref{enum:lm:lift_of_shifts:shift}}  \ar@{=>}[d] &
  \ref{enum:xx:Fhaxi_xiFha}  \ar@{=>}[d]
  \\
  \ref{enum:xx:ha__ha_xi_z}   \ar@{=>}[r]   &
  \ref{enum:xx:Fha_Fhaxi_z}   \ar@{<=>}[r]^{\ref{enum:lm:lift_of_shifts:shift}}  &
  \ref{enum:xx:Fhaxi_xiFha_z}
}
\end{equation}

\begin{enumerate}[leftmargin=*, label={\rm(\roman*)}]
\item\label{enum:impl:non_per}
If $\pz$ is non-fixed and non-periodic point for $\hflow$, then~\ref{enum:xx:Fha_Fhaxi_z}$\Rightarrow$\ref{enum:xx:ha__ha_xi_z}.
\item\label{enum:impl:lift_h}
If $\tXman$ is path connected and $\hflow_{\beta}$ is a lifting of some continuous map $\dif:\Xman\to\Xman$, that is $\dif\circ\prj = \prj\circ\hflow_{\beta}$, then \ref{enum:xx:Fhaxi_xiFha_z}$\Rightarrow$\ref{enum:xx:Fhaxi_xiFha}, whence all conditions in the right square of~\eqref{equ:diag_implications} are equivalent.
\item\label{enum:impl:loc_uni}
If $\tXman$ is path connected, and $\hflow$ has no fixed points and admits flow box charts at each point of $\tXman$, then $\ref{enum:xx:Fha_Fhaxi}\&\ref{enum:xx:ha__ha_xi_z}\Rightarrow\ref{enum:xx:ha__ha_xi}$.
\end{enumerate}
\end{enumerate}
\end{lemma}
\begin{proof}
\ref{enum:lm:lift_of_shifts:lift}
Consider the following homotopy (with ``open ends'')
\[
  \gflow=\flow\circ(\prj\times\id_{\bR}): \tXman\times\bR \to \Xman,
  \qquad
  \gflow(\pz,\pt) = \flow(\prj(\pz),\pt),
\]
and let $\hflow:\tXman\times 0 \to \tXman$ be given by $\hflow(\pz,0) = \pz$.
Then $\hflow$ is a lifting of $\gflow|_{\tXman\times 0}$, that is $\prj\circ\hflow(\pz,0) = \prj(\pz) = \flow(\prj(\pz),0) = \gflow(\pz,0)$.
Hence $\hflow$ extends to a unique lifting $\hflow:\tXman\times\bR \to \tXman$ of $\gflow$.
One easily checks that this lifting is a flow on $\tXman$.

\ref{enum:lm:lift_of_shifts:shift}
Let $\pz\in\tXman$ and $\pt = \alpha\circ\prj(\pz)$.
Then
\begin{align*}
  \prj \circ \hflow_{\alpha\circ\prj(\pz)}(\pz) =
  \prj \circ \hflow_{\pt}(\pz)
  \stackrel{\ref{enum:lm:lift_of_shifts:lift}}{=} \flow_{\pt}\circ \prj(\pz) =
  \flow_{\alpha\circ\prj(\pz)}\circ \prj(\pz) =
  \flow_{\alpha}\circ \prj(\pz).
\end{align*}

\ref{enum:lm:lift_of_shifts:comm}
Notice that the map $\hflow':\tXman\times\bR\to\tXman$ defined by $\hflow'_{\pt} = \xi\circ \hflow_{\pt}\circ\xi^{-1}$ is also a flow on $\tXman$.
Moreover,
\[
\prj \circ \hflow'_{\pt} =
\prj \circ\xi\circ \hflow'_{\pt}\circ\xi^{-1} =
\prj \circ \hflow_{\pt}\circ\xi^{-1} =
\flow_{\pt}\circ \prj \circ\xi^{-1} =
\flow_{\pt}\circ \prj.
\]
Thus $\hflow'$ and $\hflow$ are two liftings of $\flow$ which coincide at $\pt=0$, and therefore $\hflow' = \hflow$ by uniqueness of liftings.
Therefore for any continuous function $\halpha:\tXman\to\bR$ and $\pz\in\tXman$ we have, (cf.~\cite[Eq.~(6.8)]{Maksymenko:TA:2020})
\begin{align*}
\hflow_{\halpha} \circ\xi(\pz) &=
\hflow\bigl( \xi(\pz), \halpha(\xi(\pz))\bigr) =
\hflow_{\halpha(\xi(\pz))} \circ \xi(\pz)  \\
&= \xi \circ \hflow_{\halpha\circ\xi(\pz)}(\pz) =
\xi \circ \hflow_{\halpha\circ\xi}(\pz).
\end{align*}

\ref{enum:lm:lift_of_shifts:implications}
The implication in Diagram~\eqref{equ:diag_implications} are trivial.
Assume that $\tXman$ (and therefore $\Xman$) are path connected.

\ref{enum:impl:non_per}
If $\pz$ is non-fixed and non-periodic, then $\hflow(\pz,a) =\hflow(\pz,b)$ implies that $a=b$ for any $a,b$.
In particular, this hold for $a=\beta(\pz)$ and $b=\beta\circ\xi(\pz)$.

\ref{enum:impl:lift_h}
Suppose that $\hflow_{\beta}$ is a lifting of some continuous map $\dif:\Xman\to\Xman$.
Then $\prj\circ \hflow_{\beta}\circ\xi = \dif\circ\prj\circ\xi = \dif\circ\prj$ and $\prj\circ \xi \circ \hflow_{\beta} = \prj\circ\hflow_{\beta} = \dif\circ\prj$, i.e.~both $\hflow_{\beta}\circ\xi$ and $\xi\circ\hflow_{\beta}$ are liftings of $\dif$.

Now if~\ref{enum:xx:Fhaxi_xiFha_z} holds, i.e. $\hflow_{\beta}\circ\xi(\pz) = \xi\circ\hflow_{\beta}(\pz)$ at some point $\pz\in\tXman$, then these liftings must coincide on all of $\tXman$, which means condition~\ref{enum:xx:Fhaxi_xiFha}.

\ref{enum:impl:loc_uni}
Suppose $\hflow$ has no fixed points and conditions~\ref{enum:xx:Fha_Fhaxi} and~\ref{enum:xx:ha__ha_xi_z} hold, that is $\hflow(\py,\beta(\py)) = \hflow(\py,\beta\circ\xi(\py))$ for all $\py\in\tXman$ and $\beta(\pz)=\beta\circ\xi(\pz)$ for some $\pz\in\tXman$.
Notice that the set $\Aman = \{ \py\in\tXman \mid \beta(\py)=\beta\circ\xi(\py)\}$ is close.
Moreover, by the local uniqueness of shift-functions (Corollary~\ref{cor:local_uniqueness}) $\Aman$ is also open.
Since $\tXman$ is connected, $\Aman$ is either $\varnothing$ or $\tXman$.
But $\pz\in\Aman$, whence $\Aman=\tXman$, i.e. condition~\ref{enum:xx:ha__ha_xi} holds.
\end{proof}

Let $\prj:\tXman\to\Xman$ be a \myemph{regular} covering map with path connected $\tXman$ and $\Xman$, $G$ be the group of covering transformation, $\flow:\Xman\times\bR\to\Xman$ be a flow on $\Xman$, and $\hflow:\tXman\times\bR\to\tXman$ be its lifting as in Lemma~\ref{lm:lift_of_shifts}\ref{enum:lm:lift_of_shifts:lift}.
\begin{corollary}\label{lm:cov_shift_func}{\rm(c.f.~\cite[Lemma~6.3]{Maksymenko:TA:2020})}
Suppose all orbits of $\hflow$ are non-fixed and non-closed and $\hflow$ has flow box charts at all points of $\tXman$.
Let also $\dif:\Xman\to\Xman$ be a continuous map admitting a lifting $\hdif:\tXman\to\tXman$, i.e. $\prj\circ\hdif=\dif\circ\prj$, such that $\hdif$ leaves invariant each orbit $\gamma$ of $\hflow$ and commutes with each $\xi\in G$.
Then there exist a unique continuous function $\alpha: \Xman\to\bR$ such that $\dif =\flow_{\alpha}$.

Again if $\tXman$, $\Xman$, $\prj$, $\dif$, $\flow$ and its flow box charts are $\Cr{r}$, $(1\leq r\leq \infty)$, then $\alpha$ is also $\Cr{r}$.
\end{corollary}
\begin{proof}
Due to assumptions on $\hflow$ we get from Corollary~\ref{cor:shift_func_all_orb_nonclosed} that there exists a unique continuous function $\beta:\tXman\to\bR$ such that $\hdif = \hflow_{\beta}$.
Let $\xi\in G$.
Since $\hdif$ commutes with $\xi$, i.e. condition~\ref{enum:xx:Fhaxi_xiFha} of Lemma~\ref{lm:lift_of_shifts} holds, we obtain the following implications:
\[
\xymatrix@R=2ex{
   \ref{enum:xx:Fhaxi_xiFha}       \ar@{<=>}[r] &
   \ref{enum:xx:Fha_Fhaxi}         \ar@{=>}[r] &
   \ref{enum:xx:Fha_Fhaxi_z}       \ar@{=>}[r]^-{\ref{enum:impl:non_per}} &
   \ref{enum:xx:ha__ha_xi_z}       \ar@{=>}[r]^-{\ref{enum:impl:loc_uni}} &
   \text{\ref{enum:xx:ha__ha_xi}}.
}
\]
meaning that $\beta\circ\xi=\beta$.
Thus $\beta$ is invariant with respect to all $\xi\in G$, and therefore it induces a unique function $\alpha:\Xman\to\bR$ such that $\beta = \alpha\circ\prj$.
Since $\prj$ is a local homeomorphism, it follows that $\alpha$ is continuous.
Moreover, by Lemma~\ref{lm:lift_of_shifts}\ref{enum:lm:lift_of_shifts:shift}, $\hdif = \hflow_{\beta} = \hflow_{\alpha\circ\prj}$ is a lifting of $\flow_{\alpha}$.
But $\hdif$ is also a lifting of $\dif$, whence $\dif = \flow_{\alpha}$.

Statements about smoothness of $\alpha$ follows from the corresponding smoothness parts of used lemmas.
We leave the details for the reader.
\end{proof}

\section{Self maps of the circle}\label{sect:circle_maps}
Let $\Circle = \{ \pz \in \bC \mid |z| = 1\}$ be the unit circle in the complex plane, $\prj\colon \bR\to \Circle$ be the universal covering map defined by $\prj(\ps)=e^{2\pi i \ps}$ and $\xi(\ps)=\ps+1$ be a diffeomorphism of $\bR$ generating the group of covering slices $\bZ$.

Denote by $\CCircle{k}$, $k\in\bZ$, the set of all continuous maps $\dif:\Circle\to \Circle$ of degree $k$, i.e. maps homotopic to the map $\pz\mapsto \pz^k$.
Then $\{ \CCircle{k} \}_{k\in\bZ}$ is a collection of all path components of $\Cont{\Circle}{\Circle}$ with respect to the compact open topology.

For a map $\dif:\Xman\to\Xman$ it will be convenient to denote the composition $\underbrace{\dif\circ\cdots\circ\dif}_{n}$ by $\dif^n$ for $n\in\bN$.
A point $\px\in\Xman$ is \myemph{fixed} for $\dif$, whenever $\dif(\px)=\px$.

\begin{lemma}\label{lm:lift_dif_circle}
Let $\dif:\Circle\to \Circle$ be a continuous map and $\hdif_0\colon\bR\to\bR$ be any lifting of $\dif$ with respect to $p$, i.e. a continuous map making commutative the following diagram:
\[
\xymatrix{
  \bR \ar[d]_-{p}  \ar[r]^-{\hdif_0} & \bR \ar[d]^-{p}\\
  \Circle \ar[r]^-{\dif}               & \Circle
}
\]
that is $p\circ\hdif_0 = \dif\circ p$, or $e^{2\pi i\, \hdif_0(\ps)} = \dif(e^{2\pi i\ps})$ for $\ps\in\bR$.
For $a\in\bZ$ define the map $\hdif_a:\bR\to\bR$ by $\hdif_a=\xi^{a}\circ \hdif_0$, that is $\hdif_a(\ps)=\hdif_0(\ps)+a$.
Then the following conditions are equivalent:
\begin{enumerate}[leftmargin=*, label={\rm(\alph*)}, itemsep=1ex]
\item\label{enum:char:deg:a} $\dif\in\CCircle{k}$;
\item\label{enum:char:deg:b} $\hdif_0\circ\xi = \xi^k\circ\hdif = \hdif_k$, that is $\hdif(\ps+1) = \hdif(\ps) + k$ for all $t\in\bR$.
\end{enumerate}
Moreover, let $k$ be the degree of $\dif$.
Then
\begin{enumerate}[leftmargin=*, label={\rm(\roman*)}, itemsep=1ex]
\item\label{enum:lm:lift_dif_circle:fixed_pt}
$\dif$ has at least $|k-1|$ fixed points;

\item\label{enum:lm:lift_dif_circle:all_liftings}
$\{\hdif_{ak}\}_{a\in\bZ}$ is the collection of all possible liftings of $\dif$;

\item\label{enum:lm:lift_dif_circle:composition}
$\hdif_a\circ \hdif_b = \xi^{a+kb} \circ \dif^2$ for any $a,b\in\bZ$;

\item\label{enum:lm:lift_dif_circle:square}
if $k=-1$, then $\hdif_a^2 = \hdif_0^2$ for all $a\in\bZ$, in other words, \myemph{for any lifting $\hdif_a$ of $\dif_0$ its square $\hdif_a^2$ does not depend on $a$}.
Moreover, \myemph{if $\Aman$ is the set of fixed points of $\dif^2$, then $\prj^{-1}(A)$ is the set of fixed points of $\hdif_0^2$}.
\end{enumerate}
\end{lemma}
\begin{proof}
All statements are easy.
Statement~\ref{enum:lm:lift_dif_circle:fixed_pt} is a consequence of intermediate value theorem, and~\ref{enum:lm:lift_dif_circle:all_liftings} follows from~\ref{enum:char:deg:b}.

\ref{enum:lm:lift_dif_circle:composition}
$\hdif_a\circ \hdif_b =
\xi^a \circ \hdif_0 \circ \xi^b\circ \hdif_0  =
\xi^a \circ \xi^{kb} \circ \hdif_0^2  = \xi^{a + kb} \circ \hdif_0^2$.

\ref{enum:lm:lift_dif_circle:square}
If $k=-1$ then, by~\ref{enum:lm:lift_dif_circle:composition}, $\hdif_a\circ \hdif_a = \xi^{a - a} \circ \hdif_0^2 = \hdif_0^2$.

Hence one can put $\gdif := \hdif_0^2 = \hdif_a^2$ and this map does not depend on $a\in\bZ$.
Let also $\tilde{\Aman}$ be the set of fixed points of $\gdif$.
We have to show that $\tilde{\Aman} = \prj^{-1}(\Aman)$.

Let $\ps \in \tilde{\Aman}$, and $\pz = \prj(\ps)$.
Then $\ps = \gdif^2(\ps) = \hdif_a^2(\ps)$ implies that
\[ \pz = \prj(\ps) = \prj\circ\hdif_a^2(\ps) = \dif^2\circ \prj(\ps) = \dif^2(\pz), \]
so $\pz\in\Aman$, that is $\prj(\tilde{\Aman}) \subset \Aman$ and thus $\tilde{\Aman} \subset \prj^{-1}(\Aman)$.

Conversely, let $\pz\in\Aman$ and $\ps\in\bR$ be such that $\pz = \prj(\ps)$.
Then there exists a unique lifting $\hdif_a$ of $\dif$ such that $\hdif_a(\ps) = \ps$.
But then $\gdif(\ps) = \hdif_a^2(\ps) = \ps$, whence $\ps\in\tilde{\Aman}$.
\end{proof}

The following example shows that the effect described in the statement~\ref{enum:lm:lift_dif_circle:square} of Lemma~\ref{lm:lift_dif_circle} includes the \myemph{rigidity} property of reflections of the circle mentioned in the introduction.
\begin{example}
Let $\dif(\pz) = \overline{\pz e^{-2\pi \phi}} e^{2\pi \phi} = \bar{z} e^{2\phi}$ be a reflection of the complex plane with respect to the line passing though the origin and constituting an angle $\phi$ with the positive direction of $x$-axis.
Then $\dif$ is an involution preserving the unit circle and the restriction $\dif|_{\Circle}:\Circle \to \Circle$ is a map of degree $-1$ belonging to $SO^{-}(2)$.
Moreover, each its lifting $\dif_a:\bR\to\bR$ is given by $\hdif_a(\ps) = a + \phi - \ps$.
But then $\hdif_a^2 = \id_{\bR^2}$ and does not depend on $a$.
Moreover, the set of fixed points of $\hdif^2_a$ is $\bR$ which coincides with $\prj^{-1}(\Circle)$, where $\Circle$ is the set of fixed points of $\dif^2 = \id_{\Circle}$.
\end{example}

Another interpretation of the above results can be given in terms of \myemph{shift functions}.
\begin{corollary}\label{cor:shift_func_circle}
Let $\flow:\Circle\times\bR\to\Circle$ be a flow on the circle $\Circle$ having no fixed points, so $\Circle$ is a unique periodic orbit of $\flow$ of some period $\theta$.
\begin{enumerate}[wide, label={\rm(\arabic*)}, itemsep=1ex]
\item\label{cr:shift_s1:exist}
Let $\dif:\Circle\to\Circle$ be a continuous map.
Then $\dif\in\CCircle{1}$ if and only if there exists a continuous function $\alpha:\Circle\to\bR$ such that $\dif = \flow_{\alpha}$.
Such a function is not unique and is determined up to a constant summand $n\theta$ for $n\in\bZ$.
If $\flow$ and $\dif$ are $\Cr{r}$, $(0\leq r \leq \infty)$, then so is $\alpha$.

\item\label{cr:shift_s1:deg_minus_1}
For every $\dif\in\CCircle{-1}$ there exists a unique continuous function $\alpha:\Circle\to\bR$ such that
\begin{enumerate}[leftmargin=9ex, label={\rm(\alph*)}]
\item\label{enum:cor:shift_func_circle:1} $\dif^2(\pz) = \flow(\pz,\alpha(\pz))$ for all $\pz\in\Circle$;
\item\label{enum:cor:shift_func_circle:2} $\alpha(\pz)=0$ for some $\pz\in\Circle$  iff $\dif(\pz)=\pz$, (due to Lemma~\ref{lm:lift_dif_circle}\ref{enum:lm:lift_dif_circle:fixed_pt} there exists at least two such points);
\item\label{enum:cor:shift_func_circle:3} if $\flow$ and $\dif$ are $\Cr{r}$, $(0\leq r \leq \infty)$, then so is $\alpha$.
\end{enumerate}

\item\label{cr:shift_s1:homotopy}
Let $\alpha_k:\Circle\to\bR$, $k=0,1$, be two continuous functions, and for each $t\in[0;1]$ let $\alpha_t:\Circle\to\bR$ and $\dif_t:\Circle\to\Circle$ be defined by
\begin{align*}
  \alpha_{\pt} &= (1-\pt)\alpha_0 + \pt\alpha_1, &
  \dif_{\pt}(\pz) &= \pz e^{2\pi i\alpha_k(\pz)}.
\end{align*}
If $\dif_0$ and $\dif_1$ are homeomorphisms (diffeomorphsms of class $\Cr{r}$, $1\leq r\leq \infty$), then so is $\dif_{\pt}$ for each $\pt\in[0;1]$.

\end{enumerate}
\end{corollary}
\begin{proof}
Let $\hflow:\bR\times\bR\to\bR$ be a unique lifting of $\flow$ with respect to the covering $\prj$.
Then $\bR$ is a unique non-periodic orbit of $\hflow$.

\ref{cr:shift_s1:exist}
If $\alpha:\Circle\to\bR$ is any continuous function, then the map $\flow_{\alpha}$ is homotopic to the identity by the homotopy $\{\flow_{\pt\alpha}\}_{\pt\in[0;1]}$, whence $\flow_{\alpha}$ has degree $1$ (the same as $\id_{\Circle})$.

Conversely, let $\dif\in\CCircle{1}$ and let $\hdif$ be any lifting of $\dif$.
Then by Corollary~\ref{cor:shift_func_all_orb_nonclosed} there exists a unique continuous function $\beta:\bR\to\bR$ such that $\hdif = \hflow_{\beta}$.
Moreover, since $\dif$ is a map of degree $1$, we get from Lemma~\ref{lm:lift_dif_circle}\ref{enum:char:deg:b}, that $\hdif$ commutes with $\xi$, i.e. condition~\ref{enum:xx:Fhaxi_xiFha} of Lemma~\ref{lm:lift_of_shifts} holds.
Since $\hdif=\flow_{\beta}$ is a lifting of $\dif$, Lemma~\ref{lm:lift_of_shifts}\ref{enum:impl:lift_h} implies that condition~\ref{enum:xx:ha__ha_xi} of Lemma~\ref{lm:lift_of_shifts} also holds, i.e. $\beta\circ\xi=\beta$.
Hence $\beta$ induces a unique function $\alpha:\Circle\to\bR$ such that $\beta = \alpha\circ\prj$ and $\hdif = \hflow_{\alpha\circ\prj}$.
Therefore By Lemma~\ref{lm:lift_of_shifts}\ref{enum:lm:lift_of_shifts:shift}, $\hdif$ is a lifting of $\flow_{\alpha}$.
But it is also a lifting of $\dif$, whence $\dif = \flow_{\alpha}$.

\ref{cr:shift_s1:deg_minus_1}
Since $\dif\in\CCircle{-1}$, it follows that $\dif^2\in\CCircle{1}$, whence~\ref{enum:cor:shift_func_circle:1} and~\ref{enum:cor:shift_func_circle:3} directly follow from~\ref{cr:shift_s1:exist}.

\ref{enum:cor:shift_func_circle:2}
Denote by $\Aman$ the set of fixed points of $\dif^2$.
We should prove that $\Aman = \alpha^{-1}(0)$.
Evidently, if $\alpha(\py)=0$, then \[ \dif^2(\py) = \flow(\py, \alpha(\py)) = \dif^2(\py) = \flow(\py, 0)=\py.\]

Conversely, let $\py\in\Aman$, so $\dif(\py) = \py$, and let $\px\in\bR$ be any point with $\prj(\px) = \py$.
Then there exists a unique lifting $\hdif$ of $\dif$ with $\hdif(\px)=\px$ and by Lemma~\ref{lm:lift_dif_circle}\ref{enum:lm:lift_dif_circle:square}, $\prj^{-1}(\Aman)$ is the set of fixed points of $\hdif^2$.
Moreover, since $\hdif^2$ has degree $1$, we get from~\ref{cr:shift_s1:exist} that $\hdif^2 = \hflow_{\alpha\circ\prj}$.
Since $\px$ is a fixed point of $\hdif^2$ as well, we must have that $0 = \alpha\circ\prj(\px) = \alpha(\py)$.

\ref{cr:shift_s1:homotopy}
Consider two flows of $\bR$ and $\Circle$ respectively:
\begin{align*}
  &\hflow:\bR\times\bR\to\bR,        &\hflow(\px,\pt) &= \px+\pt,\\
  &\flow:\Circle\times\bR\to\Circle, &\flow(\pz, \pt) &= \pz e^{2\pi i \pt}.
\end{align*}
Evidently, $\hflow$ is a lifting of $\flow$, and $\dif_{\pt} = \flow_{\alpha_{\pt}}$.
Then by Lemma~\ref{lm:lift_of_shifts}\ref{enum:lm:lift_of_shifts:shift} the map
\[
  \hdif_{\pt}:=\hflow_{\alpha_{\pt}\circ\prj}:\bR\to\bR,
  \qquad
  \hdif_{\pt}(\px)= \px + (1-\pt)\alpha_0(\px) + \pt\alpha_1(\px)
\]
is a lifting of $\dif_{\pt}$.
Evidently, $\hdif_{\pt} = (1-\pt)\hdif_0 + \pt\hdif_1$.
Since $\dif_0$ and $\dif_1$ are homeomorphisms (diffeomorphisms of class $\Cr{r}$), it follows that so are $\hdif_0$ and $\hdif_1$ are homeomorphisms.
Therefore so are their \myemph{convex linear combination} $\hdif_{\pt}$ and the induced map $\dif_{\pt}$.
\end{proof}

\section{Flows without fixed points}\label{sect:flows_no_fixpt}
Let $\Xman$ be a Hausdorff topological space and $\flow\colon\Yman\times\bR\to\Yman$ a continuous flow on an open subset $\Yman\subset\Xman\times\Circle$ satisfying the following conditions:{\em
\begin{enumerate}[leftmargin=*, label={$(\Phi\arabic*)$}]
\item\label{cond:star:1}
the orbits of $\flow$ are exactly the connected components of the intersections $(\px\times\Circle)\cap\Yman$ for all $\px\in\Xman$.
\item\label{cond:star:flow_box}
$\flow$ admits flow box charts at each point $(\py,\ps)\in\Yman$;

\item\label{cond:star:B_is_dense}
the set $\Bman = \{ \px \in \Xman \mid \px\times\Circle \subset \Yman \}$ is dense in $\Xman$;

\item\label{cond:star:finite_comp}
for each $\px\in\Xman$, the intersection $(\px\times\Circle) \cap \Yman$ has only finitely many connected components (being by~\ref{cond:star:1} orbits of $\flow$).
\end{enumerate}}
Thus every orbit of $\flow$ is either $\px\times\Circle$ or some arc in $\px\times\Circle$, and, in particular, $\flow$ has no fixed points.
Also condition~\ref{cond:star:B_is_dense} implies that the set of periodic orbits of $\flow$ is dense in $\Xman\times\Circle$.
Then the complement to $\Bman$:
\[
\Aman := \Xman\setminus\Bman = \{ \px \in \Xman \mid \px\times\Circle \not\subset \Yman \}
\]
consists of $\px\in\Xman$ for which $\px\times\Circle$ contains a non-closed orbit of $\flow$.

It will also be convenient to use the following notations for $\px\in\Xman$:
\begin{align*}
  \Gx  &:= (\px\times\Circle) \cap \Yman, &
  \tGx &:= \prj^{-1}(\Gx) = (\px\times\bR) \cap \tYman.
\end{align*}

\begin{example}
Let $\Xman$ be a smooth manifold, and $\gfld$ be a vector field on $\Xman\times\Circle$ defined by $\gfld(\px,\ps) = \ddx{}{\ps}$, so its orbits are the circles $\px\times\Circle$.
Let $\Yman \subset \Xman\times\Circle$.
Then it is well known that there exists a non-negative $\Cinfty$ function $\alpha:\Xman\times\Circle \to[0;+\infty)$ such that $\Yman = (\Xman\times\Circle)\setminus\alpha^{-1}(0)$.
Define another vector field $\fld$ on $\Xman\times\Circle$ by $\fld = \alpha\gfld$.
Let also \[ \flow: (\Xman\times\Circle) \times\bR \to \Xman\times\Circle \]
be the flow on $\Xman\times\Circle$ generated by $\fld$.
Then $\Yman$ is the set of non-fixed points of $\flow$ and the induced flow on $\Yman$ satisfies conditions~\ref{cond:star:1} and~\ref{cond:star:flow_box}.
Another two conditions~\ref{cond:star:B_is_dense} and~\ref{cond:star:finite_comp} depend only on a choice of $\Yman$.
For instance they will hold if the complement $(\Xman\times\Circle)\setminus\Yman$ is finite.
\end{example}

\begin{theorem}\label{th:shift_func_XS1_L}
Suppose $\flow$ satisfies~\ref{cond:star:1}-\ref{cond:star:finite_comp}.
Let also $\dif:\Yman\to\Yman$ be a homeomorphism having the following properties.

\begin{enumerate}[leftmargin=*, label={\rm(\arabic*)}, itemsep=1ex]
\item\label{enum:shift_func_XS1_L:h_pres_levels}
$\dif(\Gx) \subset \px\times\Circle$ for each $\px\in\Xman$, so $\dif$ preserves the first coordinate, and in particular, leaves invariant each periodic orbit of $\flow$, thought it may interchange non-periodic orbits contained in $\px\times\Circle$.

\item\label{enum:shift_func_XS1_L:h_pres_periodic_orbits}
For each $\px\in\Bman$ the restriction $\dif: \px\times\Circle \to \px\times\Circle$ has degree $-1$ as a self-map of a circle.
\end{enumerate}
Then there exists a unique continuous function $\alpha:\Yman\to\bR$ such that
\begin{enumerate}[label={\rm(\alph*)}, itemsep=1ex]
\item\label{enum:shift_func_XS1_L:1}
$\dif^2(\px,\ps) = \flow(\px,\ps,\alpha(\px,\ps))$ for all $(\px,\ps)\in\Yman$, in particular $\dif^2$ preserves \myemph{every} orbit of $\flow$;

\item\label{enum:shift_func_XS1_L:2}
$\alpha(\px,\ps)=0$ iff $\dif(\px,\ps)=(\px,\ps)$;

\item\label{enum:shift_func_XS1_L:3}
if in addition $\Xman$ is a manifold of class $\Cr{r}$, $(0\leq r \leq \infty)$, $\flow$ is $\Cr{r}$ and admits $\Cr{r}$ flow box charts, and $\dif$ is $\Cr{r}$, then $\alpha$ is $\Cr{r}$ as well.
\end{enumerate}
\end{theorem}
\begin{proof}
Let $\prj\colon\Xman\times\bR\to\Xman\times\Circle$ be the infinite cyclic covering of $\Xman\times\Circle$ defined by $\prj(\px,\ps)=(\px, e^{2\pi i \ps})$ and $\xi(\px,\ps)=(\px, \ps+1)$ be a diffeomorphism of $\Xman\times\bR$ generating the group $\bZ$ of covering slices.
In particular, $\prj\circ\xi=\prj$.
Denote $\tYman= \prj^{-1}(\Yman)$
Then the restriction $\prj\colon\tYman\to\Yman$ is a covering of $\Yman$.

Let $\hflow:\tYman\times\bR\to\tYman$ be the lifting of the flow $\flow$, so
\[
  \prj\circ\hflow(\py,\pt) = \flow(\prj(\py),\pt), \qquad (\py,\pt)\in \tYman\times\bR.
\]
Then the orbits of $\hflow$ are exactly the connected components of $\tGx$ for all $\px\in\Xman$.
In particular, all orbits of $\hflow$ are non-closed.
Evidently, for $\px\in\Xman$ the following conditions are equivalent:
\begin{align*}
  \bullet&~\text{$\px\times\bR$ is an orbit of $\hflow$}; &
  \bullet&~\text{$\px\times\bR \subset \tYman$};  \\
  \bullet&~\text{$\px\times\Circle$ is an orbit of $\flow$}; &
  \bullet&~\text{$\px\times\Circle \subset \Yman$, i.e. $\px\in\Bman$}.
\end{align*}

\begin{sublemma}\label{lm:prop_hdif}
Let $\hdif:\tYman\to\tYman$ be any lifting of $\dif$, that is $\prj\circ\hdif = \dif\circ\prj$.
Then the following conditions hold.
\begin{enumerate}[leftmargin=*, label={\rm(\roman*)}, itemsep=0.8ex]
\item\label{enum:shift_func_XS1_L:A:inv_x_R}
$\hdif(\tGx) \subset \tGx$ for all $\px\in\Xman$;

\item\label{enum:shift_func_XS1_L:A:hh_square_is_the_same}
if $\hdif_1$ is another lifting of $\dif$, then $\hdif_1^2 = \hdif^2$;

\item\label{enum:shift_func_XS1_L:A:decr}
for each $\px\in\Xman$ the restriction $\hdif: \tGx \to \tGx$ is strictly decreasing in the sense that if $\dif(\px, \ps) =(\px, \pt)$ for some $\ps,\pt\in\bR$, then $\ps > \pt$;

\item\label{enum:shift_func_XS1_L:A:hh2_leaves_orbits_inv}
for each $\px\in\Xman$ the restriction $\hdif^2$ leaves invariant each orbit of $\hflow$.
\end{enumerate}
\end{sublemma}
\begin{proof}
\ref{enum:shift_func_XS1_L:A:inv_x_R}
Notice that $\prj\circ\hdif(\tGx)  =  \dif\circ\prj(\tGx)   =  \dif(\Gx)  \subset  \Gx$, whence
\[
  \hdif(\tGx)
    \, \subset \,
  \prj^{-1}\bigl( \dif(\Gx) \bigr)
    \, \subset \,
  \prj^{-1}(\Gx)
    \, = \,
  \tGx.
\]

\smallskip

\ref{enum:shift_func_XS1_L:A:hh_square_is_the_same}
Let $\px\in\Bman$, so $\px\times\Circle \subset \Yman$ is an orbit of $\flow$.
Then the restriction $\prj|_{\px\times\bR}:\px\times\bR \to \px\times\Circle$ is a universal covering map, $\dif|_{\px\times\Circle}:\px\times\Circle \to \px\times\Circle$ is a map of degree $-1$, and $\hdif|_{\px\times\bR}, \hdif_1|_{\px\times\bR}:\px\times\bR \to \px\times\bR$ are two lifting of $\dif|_{\px\times\Circle}$.
Hence by Lemma~\ref{lm:lift_dif_circle}\ref{enum:lm:lift_dif_circle:square} $\hdif^2|_{\px\times\bR} = \hdif_1^2|_{\px\times\bR}$.
Thus $\hdif^2 = \hdif_1^2$ on $\Bman\times\bR$.
But by property~\ref{cond:star:B_is_dense}, $\Bman$ is dense in $\Xman$, whence $\Bman\times\bR$ is dense in $\Xman\times\bR$.
Since $\Xman$ is Hausdorff, it follows that $\hdif^2 = \hdif_1^2$ on all of $\Xman\times\bR$.

\smallskip

\ref{enum:shift_func_XS1_L:A:decr}
If $\px\in\Bman$, then $\dif|_{\px\times\Circle}:\px\times\Circle \to \px\times\Circle$ is a map of degree $-1$, whence $\hdif|_{\px\times\bR}:\px\times\bR \to \px\times\bR$ reverses orientation of $\px\times\bR$, i.e. is strictly decreasing.

Suppose $\px\in \Xman\setminus\Bman$, i.e. $\px\times\bR\not\subset\tYman$.
We need to show that $\hdif|_{\tGx}:\tGx\to\tGx$ is strictly decreasing, that is if $(\px,\ps_0), (\px,\ps_1)\in\tGx$ are two distinct points with $\ps_0<\ps_1$, $(\px,\pt_0) = \hdif(\px,\ps_0)$, and $(\px,\pt_1) = \hdif(\px,\ps_1)$, then $\pt_0 > \pt_1$.

Since $\hdif$ is a homeomorphism, it follows that $\pt_0 \not= \pt_1$.
Suppose that $\pt_0 < \pt_1$.
Then there exist $a>0$, and two open neighborhoods $\Vman \subset \Uman$ of $\px$ in $\Xman$ such that
\begin{gather}
\label{equ:th_Vsi__U_pti}
\hdif\bigl( \Vman \times \ps_i \bigr) \ \subset \ \Uman \times(\pt_i - a; \pt_i + a), \quad  i=0,1, \\
\label{equ:t0pa_lt_t1ma}
\pt_0 + a \ < \ \pt_1 - a.
\end{gather}
Due to property~\ref{cond:star:B_is_dense}, the set $\Bman$ is dense in $\Xman$, so there exists a point $\py\in\Bman\cap\Vman \not=\varnothing$.

Then on the one hand $\py\times\bR \subset\tYman$ is an orbit of $\hflow$, and $\hdif:\py\times\bR \to \py\times\bR$ reverses orientation, so if $(\py, \pt'_i) = \hdif(\py,\ps_i)$, $i=0,1$, then $\pt'_0 > \pt'_1$.

On the other hand, due to~\eqref{equ:th_Vsi__U_pti} $\pt'_i \in (\pt_i - a; \pt_i + a)$, whence by~\eqref{equ:t0pa_lt_t1ma}:
\[
  \pt'_0 \ < \ \pt_0 + a \ < \ \pt_1 - a \ <  \ \pt'_1
\]
which gives a contradition.
Hence $\pt_0 > \pt_1$.

\smallskip

\ref{enum:shift_func_XS1_L:A:hh2_leaves_orbits_inv}
If $\px\in\Bman$, then $\tGx = \px\times\bR$ is an orbit of $\hflow$ and by~\ref{enum:shift_func_XS1_L:A:inv_x_R} it is invariant with respect to $\hdif$.
Hence it is also invariant with respect to $\hdif^2$.

Suppose $\px\in\Xman\setminus\Bman$.
Then by property~\ref{cond:star:finite_comp}, if $\Gx = (\px\times\Circle) \cap \Yman$ consists of finitely many connected components $\Ji{0},\ldots,\Ji{n-1}$ for some $n$ enumerated in the cyclical order along the circle $\px\times\Circle$.
This implies that $\tGx$ is a disjoint union of countably many open intervals $\Ji{i}$, $(i\in\bZ)$, being orbits of $\flow$ which can be enumerated so that $\xi(\tJi{k}) = \tJi{k+n}$ and $\prj(\tJi{k}) = \Ji{k\bmod n}$.

Since $\hdif$ is a strictly decreasing homeomorphism of $\tGx$, it follows that there exists $a\in\bZ$ such that $\hdif(\tJi{k}) = \tJi{a-k}$.
Hence
\[
  \hdif^2(\tJi{k}) = \hdif(\tJi{a-k}) = \tJi{a-(a-k)} = \tJi{k}.
  \qedhere
\]
\end{proof}

Now we can deduce our Theorem from Lemma~\ref{lm:prop_hdif}.

\ref{enum:shift_func_XS1_L:1}
By Corollary~\ref{lm:cov_shift_func} applied to $\dif^2$ there exists a unique continuous function $\alpha:\Yman\to\bR$ such that
\begin{align*}
  \hdif^2(\px,\ps)&=\hflow(\px,\ps,\alpha\circ\prj(\px,\ps)), &
  \dif^2(\px,\ps)&=\flow(\px,\ps,\alpha(\px,\ps)).
\end{align*}

\ref{enum:shift_func_XS1_L:2}
Let $(\px,\ps)\in\Xman\times\Circle$.
If $\px\in\Bman$, then $\px\times\Circle \subset \Yman$ is an orbit of $\flow$ and by Corollary~\ref{cor:shift_func_circle}\ref{enum:cor:shift_func_circle:2} $\alpha(\px,\ps)=0$ iff $\dif^2(\px,\ps) = (\px,\ps)$.

Suppose $\px\in\Xman\setminus\Bman$.
Then $(\px,\ps)$ is a non-periodic point of $\flow$, whence $\dif^2(\px,\ps) = \flow(\px,\ps,\alpha(\px,\ps)) = (\px,\ps)$ is possible if and only if $\alpha(\px,\ps)=0$.

\ref{enum:shift_func_XS1_L:3}
Smoothness properties of $\alpha$ follows similarly to the statement~\ref{enum:cor:shift_func_circle:3} of Corollary~\ref{lm:cov_shift_func}.
\end{proof}

\section{Polar coordinates}\label{sect:polar_coord}
Let $\tHsp = \bR\times[0;+\infty)$, $\oHsp = \bR\times(0;+\infty)$, and
\[
  \prj: \tHsp  \to \bC \equiv \bR^2, \qquad
  \prj(\rho, \phi) = \rho e^{2\pi i\phi} = (\rho\cos\phi, \rho\sin\phi),
\]
be the \myemph{infinite branched coveging map defining polar coordinates}.

\begin{lemma}\label{lm:lift}{\rm(\cite[Lemma~11.1]{Maksymenko:TA:2020})}
Let $\Uman\subset\bC$ be an open neighborhood of $0$, and $\dif:\Uman\to\bC$ a $\Cr{r}$, $1\leq r\leq\infty$, smooth embedding with $\dif(0)=0$.
Then there exists a $\Cr{r-1}$ embedding $\hdif:\prj^{-1}(\Uman) \to \tHsp$ such that $\prj\circ\hdif=\dif\circ\prj$.

Suppose, in addition, that the Jacobi matrix $J(\dif,0)$ of $\dif$ at $0$ is orthogonal.
Then there exists $a\in\bR$ such that for each $\ps\in\bR$
\[
\hdif(0,\ps) =
\begin{cases}
  (0,a+\ps) \ \text{if} \ J(\dif,0)\in\SO(2), \\
  (0,a-\ps) \ \text{if} \ J(\dif,0)\in\SO^{-}(2).
\end{cases}
\]
Hence in the second case (when $\dif$ reverses orientation) $\hdif^2(0,\ps) = (0,\ps)$, that is $\hdif$ is always fixed on $0\times\bR$.
\end{lemma}

Let $\func:\bR^2\to\bR$ be a homogeneous polynomial without multiple factors and having degree $k \geq 2$ and $\fld = -\ddx{\func}{\py} \ddx{}{\px} + \ddx{\func}{\px} \ddx{}{\py}$ be the Hamiltonial vector field of $\func$.
Since the restriction $\prj:\Int(\mathbb{H})\to\bR^2\setminus 0$ is a infinite cyclic covering map, $\fld$ induces a vector field $\hfld$ on $\oHsp$.
One can even obtain precise formulas for $\hfld$ (see~\cite[\S4.2, Corolary~4.4]{Maksymenko:hamv2}):
\[
  \hfld(r,\phi) =
  \bigl(\ddx{}{ \rho}, \ddx{}{\phi}  \bigr)
  \begin{pmatrix}
  \cos\phi & \sin\phi \\
  -\tfrac{1}{\rho} \sin\phi & \tfrac{1}{\rho} \cos\phi
  \end{pmatrix}
  \begin{pmatrix}
    -\func'_{y}(\pz) \\
     \func'_{x}(\pz)
  \end{pmatrix},
\]
where $\pz = \prj(\rho,\phi) = \rho e^{2\pi i}$.
The latter formula can be reduced to a more simplified form.
Define the following function $q:\bC\setminus0 \to \bC$ by
\[
q(\pz) =
\frac{ - \func'_y(\pz) + i \func_x(\pz) }{\pz} =
\frac{ (- \func'_y(\pz) + i \func_x(\pz))\bar{\pz}}{|\pz|^2}.
\]
Then
\[
  \hfld(r,\phi) = \mathrm{Re}(q(\pz)) \ddx{}{\rho} + \mathrm{Im}(q(\pz)) \tfrac{1}{\rho} \ddx{}{\phi}.
\]
Since $\func$ is a homogeneous of degree $k\geq2$, it follows that $\hfld(r,\phi)$ smoothly extends to $\tHsp$.

Let $\flow$ and $\hflow$ be the local flows generated by $\fld$ and $\hfld$ respectively.
Then $\flow_t\circ\prj(\rho,\phi) = \prj\circ\hflow_t(\rho,\phi)$ whenever all parts of that identity are defined.
\begin{example}
1) Suppose $0$ is a non-degenerate local extreme of $\func$.
Then one can assume that $\func(\px,\py) = \tfrac{1}{2}(\px^2+\py^2)$.
In this case
\begin{align*}
  \fld(\px,\py) &= -\py\ddx{}{\px} + \px \ddx{}{\py}, &
  \flow(\pz, \pt) &= \pz e^{2\pi i \pt}, \\
  \hfld(\rho,\phi) &= \ddx{}{\phi}, &
  \hflow(\rho,\phi, \pt) &= (\rho, \phi+\pt).
\end{align*}

2) Let $0$ be a non-degenerate saddle, so one can assume that $\func(\px,\py) = \px\py$.
Then
\begin{gather*}
  \fld(\px,\py) = -\px\ddx{}{\px} + \py \ddx{}{\py}, \qquad
  \flow(\px,\py, \pt) = \bigl(\px e^{-\pt}, \py e^{\pt}\bigr), \\
  \hfld(\rho,\phi) = \rho\cos2\phi \ddx{}{\rho} + \sin2\phi \ddx{}{\phi}.
\end{gather*}
Notice that writing down precise formulas for $\hflow$ is a rather complicated task.

3) If $0$ is a degenerate critical point of $\func$, so $\deg\func \geq3$, then the situation is more complicated.
Notice that in this case $\hfld$ is zero on $\partial\tHsp = \bR\times 0$, whence $\hflow$ is fixed on that line.
Again the formulas for $\flow$ and $\hflow$ are highly complicated.
\end{example}

\begin{lemma}\label{lm:ext_alpha}{\rm(see~\cite[Theorem~1.6]{Maksymenko:hamv2}, \cite[Proof of Theorem~5.6]{Maksymenko:TA:2020})}
Let $\Uman\subset\bR^2$ be an open neighborhood of the origin $0\in\bR^2$, $\dif:\Uman\to\bR^2$ an embedding which preserves orbits of $\flow$, and $\alpha:\Uman\setminus 0\to\bR$ be a $\Cinfty$ function such that $\dif(\px) = \flow(\px, \alpha(\px))$ for all $\px\in\Uman\setminus 0$.
Let also $\hdif:\prj^{-1}(\Uman) \to \tHsp$ be any lifting of $\dif$ as in Lemma~\ref{lm:lift}.
Then $\alpha$ can be defined at $\px$ so that it becomes $\Cinfty$ in $\Uman$ in the following cases:
\begin{enumerate}[leftmargin=*, label={\rm(\alph*)}]
\item\label{emum:lm:ext_alpha:1} $\px$ is a non-degenerate local extreme of $\func$ or a (possibly degenerate) saddle point;
\item\label{emum:lm:ext_alpha:2} $\px$ is a degenerate local extreme of $\func$ and $\hdif$ is fixed on $\bR\times0$.
\item\label{emum:lm:ext_alpha:3}
$\px$ is a degenerate local extreme of $\func$ and there exists an open neighborhood $\Vman\subset\Uman$ and another embedding $\qdif:\Uman\to\bR^2$ such that $\qdif(\Vman) \subset \Uman$, $\qdif$ preserves orbits of $\flow$ and \myemph{reverses} their orientation, and $\dif = \qdif^2$.
\end{enumerate}
\end{lemma}
\begin{proof}
Cases~\ref{emum:lm:ext_alpha:1} and~\ref{emum:lm:ext_alpha:2} are proved in~\cite[Theorem~1.6]{Maksymenko:hamv2}, see also proof of Theorem~5.6 in~\cite{Maksymenko:TA:2020}.

\ref{emum:lm:ext_alpha:3}$\Rightarrow$\ref{emum:lm:ext_alpha:2}
Let $\widetilde{\qdif}:\prj^{-1}(\Uman) \to \tHsp$ be any lifting of $\qdif$ as in Lemma~\ref{lm:lift}.
Then $\widetilde{\qdif}^2$ is a lifting of $\qdif^2 = \dif$.
Moreover, since $\widetilde{\qdif}$ reverse orientation of orbits, we get from Lemma~\ref{lm:lift} that $\widetilde{\qdif}^2$ is fixed on $\bR\times0$.
Hence condition~\ref{emum:lm:ext_alpha:2} holds.
\end{proof}

\section{Chipped cylinders of a map $\func\in\FMP$}\label{sect:chip_cyl}
Let $\func\in\FMP$.
In what follows we will use the following notations.
\begin{enumerate}[wide, label={\rm(\roman*)}, itemsep=1ex]
\item
$\Kman_1,\ldots,\Kman_{\vk}$ denote all the critical leaves of $\func$, and
\[ \KK = \mathop{\cup}\limits_{i=1}^{\vk}\Kman_i.\]

\item
Let $\regN{\Kman_i}$, $i=1,\ldots,\vk$, be an $\func$ regular neighborhood of $\Kman_i$ choosen so that $\regN{\Kman_i} \cap \regN{\Kman_j}=\varnothing$ for $i\not=j$.

\item
Let also $\Lman_1,\ldots,\Lman_{\vl}$ be all the connected components of $\Mman\setminus \KK$;

\item
For each $i=1,\ldots,\vl$ let
\[ \chp_i = \overline{\Lman_i} \setminus\fSing. \]
Then there exist a finite subset $\Qman_i \subset \{-1,1\} \times \Circle$, and an immersion $\phi_i:\bigl([-1,1]\times\Circle\bigr) \setminus \Qman_i \to \chp_i$ and a $\Cinfty$ embedding $\eta:[0,1]\to\Pman$ such that the following diagram is commutative:
\[
\xymatrix{
  \bigl([0;1]\times\Circle\bigr) \setminus \Qman_i \ar[rr]^-{\phi} \ar[d]_{\prj_1} && \chp_i \ar[d]^-{\func}   \\
   [0;1] \ar[rr]^-{\eta} && \Pman
}
\]
where $\prj_1$ is the projection to the first coordinate.
Notice that $\phi$ can be non-injective only at points of $\{-1,1\}\times\Circle$ and this can happens only when $\Pman=\Circle$, see Example~\ref{exmp:chipped_cyl} and Figure~\ref{fig:chp:examples}d) below.

We will call $\chp_i$ a \myemph{chipped} cylinder of $\func$, see Figure~\ref{fig:chipped_cyl}.
\begin{figure}[ht]
\centering
\includegraphics[width=0.8\textwidth]{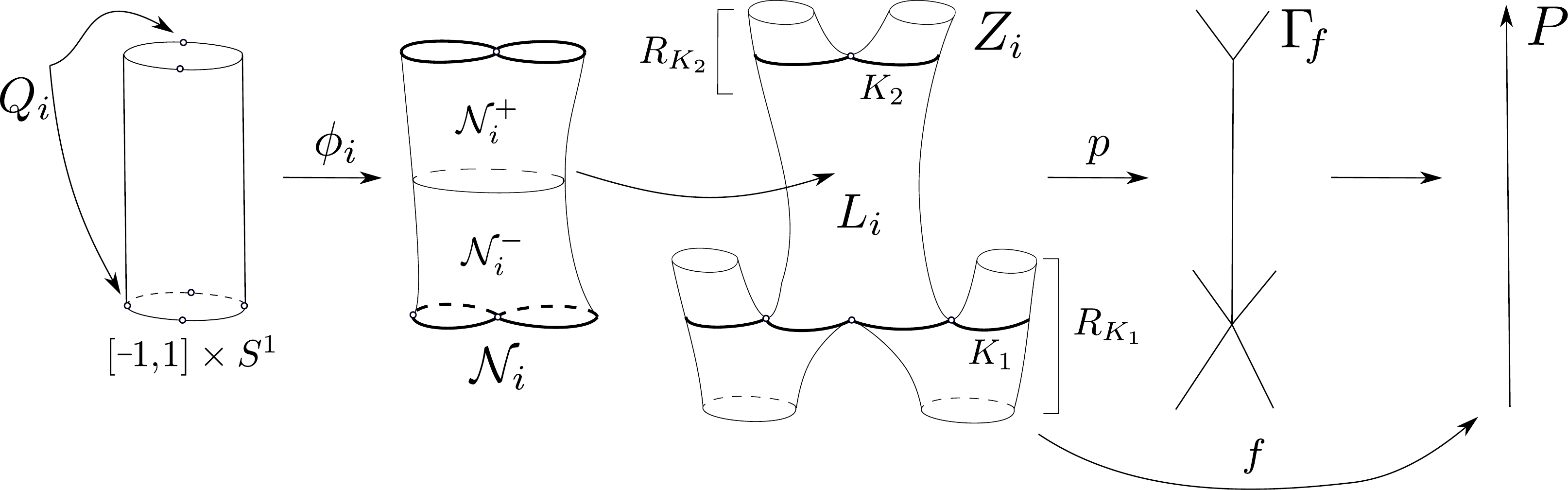}
\caption{}\label{fig:chipped_cyl}
\end{figure}

It will also be convenient to denote
\begin{gather*}
  \chp_i^{-} = \phi_i\bigl([-1;0]\times\Circle\bigr) \setminus \Qman_i\bigr),  \qquad
  \chp_i^{+} = \phi_i\bigl([0;1]\times\Circle\bigr) \setminus \Qman_i\bigr), \\
  \Int{\chp_i} = \phi_i\bigl((-1;1)\times\Circle\bigr).
\end{gather*}
We will call $\chp_i^{-}$ and $\chp_i^{+}$ \myemph{chipped half-cylinders} of $\chp_i$ and $\func$, and $\Int{\chp_i}$ the \myemph{interior} of $\chp_i$.

\item
Let also
\[
\Zman_i \, = \, \chp_i \, \bigcup \, \Bigl( \mathop{\cup}\limits_{\overline{\chp_i} \cap \Kman_j \not=\varnothing} \regN{\Kman_j} \Bigr)
\]
be the union of the chipped cylinder $\chp_i$ with $\func$-regular neighborhoods of critical leaves of $\func$ which intersect the closure $\overline{\chp_i}$.
We will call $\Zman_i$ an \myemph{$\func$-regular neighborhood} of $\chp_i$.
\end{enumerate}

\begin{example}\label{exmp:chipped_cyl}
a) Let $\func:[0,1]\times\Circle\to P$ be a map of class 
$\mathcal{F}([0,1]\times\Circle,P)$
having no critical points, see Figure~\ref{fig:chp:examples}a).
Then it has a unique chipped cylinder $\chp = [0,1]\times\Circle$, which coincides with its $\func$-regular neighborhood and $\KK=\varnothing$.

b) Let $\func:D^2\to P$ be a map of class $\mathcal{F}(D^2,P)$ having only one critical point $\pz$, see Figure~\ref{fig:chp:examples}b).
Then $\KK = \{\pz\}$, $\func$ has a unique chipped cylinder $\chp = D^2\setminus\{\pz\}$, and its $\func$-regular neighborhood is all $D^2$.

c) Let $\func:S^2\to P$ be a map of class $\mathcal{F}(S^2,P)$ having only two critical points $\pz_1$ and $\pz_2$ being therefore extremes of $\func$, see Figure~\ref{fig:chp:examples}c).
Then $\KK = \{\pz_1, \pz_2\}$, $\func$ has a unique chipped cylinder $\chp = S^2\setminus\{\pz_1,\pz_2\}$ and its $\func$-regular neighborhood is all $S^2$.

d) Let $\Mman$ be either a $2$-torus or Klein bottle with a hole and $\func:\Mman\to\Circle$ be a map of class $\mathcal{F}(M,\Circle)$ schematically shown in Figure~\ref{fig:chp:examples}d).
It has only one critical point $\pz$ and that point is a saddle, a unique critical leaf $\Kman = \KK$, and two chipped cylinders $\chp_1$ and $\chp_2$.
It follows from the Figure~\ref{fig:chp:examples}d) that $\overline{\chp_1}$ intersects only one $\Kman$ from ``both sides'', in the sense that both intersections $\overline{\chp_1^{-}}\cap\Kman$ and  $\overline{\chp_1^{+}}\cap\Kman$ are non-empty.

e) Let $\func:[0,1]\times\Circle\to\bR$ be a Morse function having one minimum $\pz$ and one saddle point $\py$ as in Figure~\ref{fig:chp:examples:all}.
Then $\func$ has two critical leaves: the point $\pz$ and a critical leaf $\Kman$ containing $\py$, and three chipped cylinders $\chp_1$, $\chp_2$, $\chp_3$.
Let $\regN{\Kman}$ be an $\func$-regular neighborhood of $\Kman$.
Then the corresponding $\func$-regular neighborhoods of chipped cylinders are the following ones:
\begin{align*}
  \Zman_1 &= \chp_1 \cup \{\pz\}\cup \regN{\Kman}, &
  \Zman_2 &= \chp_2 \cup \regN{\Kman}, &
  \Zman_3 &= \chp_3 \cup \regN{\Kman}.
\end{align*}
\end{example}

\begin{figure}[htbp!]
\begin{tabular}{ccccccc}
\includegraphics[height=1.4cm]{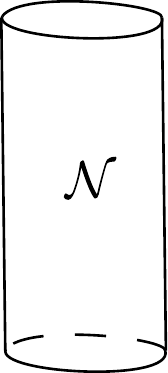} &&
\includegraphics[height=1.4cm]{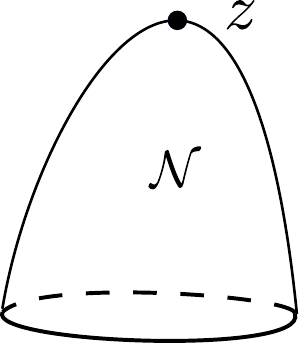} &&
\includegraphics[height=1.4cm]{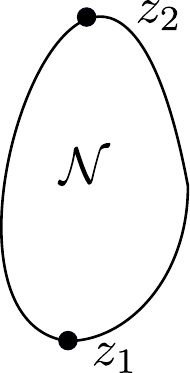} &&
\includegraphics[height=1.4cm]{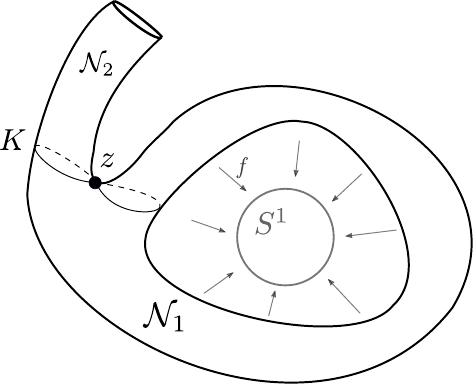} \\
a)  & \qquad &  b) &\qquad &  c) &\qquad &  d)
\end{tabular}
\caption{}\label{fig:chp:examples}
\end{figure}

\begin{figure}[htbp!]
\includegraphics[height=1.8cm]{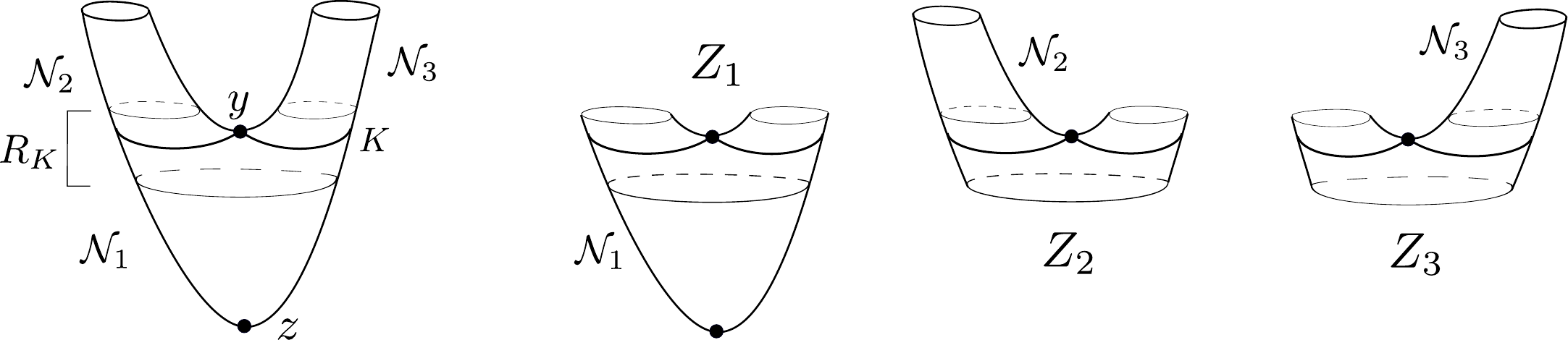}
\caption{}\label{fig:chp:examples:all}
\end{figure}

The following lemma describes simple properties of chipped cylinders.
The proof is left for the reader.
\begin{lemma}\label{lm:chp_cyl_prop}
Let $\chp$ be a chipped cylinder of $\func$, $\chp^{-}$ and $\chp^{+}$ be its chipped half-cylinders, $\Zman$ be $\func$-regular neighborhood of $\chp$, and $\prj:\Mman\to\KRGraphf$ be the projection onto the graph of $\func$.
Then the following statements hold.
\begin{enumerate}[leftmargin=*, label={\rm(\arabic*)}, itemsep=1ex]
\item
$\chp^{-}$ and $\chp^{+}$ are orientable manifolds.
Moreover, if $\Pman=\bR$, then $\chp$ is orientable as well.
However, if $\Pman=\Circle$, then it is possible to construct an example of $\func$ having non-orientable chipped cylinder, see Example~\ref{exmp:chipped_cyl}d).

\item
Each of the closures $\overline{\chp^{-}}$ and $\overline{\chp^{+}}$ intersects at most one critical leaves of $\func$, and those intersections consist of open arcs being leaves of the singular foliation $\SingFol$.

\item\label{enum:lm:chp_cyl_prop:3}
If $\chp \cap \Kman_i = \varnothing$ but $\overline{\chp} \cap \Kman_i \not= \varnothing$, then $\Kman_i$ is a critical point of $\func$ being a local extreme.

\item \label{saddle_p}
If $M$ is connected and $\overline{\chp}$ does not contain any saddle point of $f$, then $M=\overline{\chp}=\Zman$ and $f$ is one of the maps from a{\rm)}-c{\rm)} of Example \ref{exmp:chipped_cyl}.
\item
$\chp' \cap \chp'' \subset \KK$ for any two distinct chipped cylinders $\chp'$ and $\chp''$ of $\func$.

\item
Every regular leaf of $\func$ is contained in some chipped cylinder of $\func$.


\end{enumerate}
\end{lemma}

The following theorem is the principal technical tool.
Let $\Mman$ be a (possibly non-orientable) compact surface, $\func\in\FMP$, $\chp$ be a chipped cylinder of $\func$, and $\Zman = \chp \bigcup \Bigl( \mathop{\cup}\limits_{\overline{\chp}\cap\Kman_j \not=\varnothing} \regN{\Kman_j} \Bigr)$ be its $\func$-regular neighborhood of $\chp$.
\begin{theorem}\label{th:shift_func_on_chip_nbh}
{\rm 1)} Let $\Vman\subset\chp$ be a regular leaf of $\func$, and $\dif\in\Stabilizer{\func|_{\Zman}}$ a diffeomorphism of $\Zman$ such that $\dif(\Vman)=\Vman$ and $\dif$ \myemph{reverses} orientation of $\Vman$.
Then every leaf of the singular foliation $\SingFol$ in $\Zman$ is $(\dif^2, +)$-invariant, i.e. $\dif^2$ fixes all critical points of $\func$ in $\Zman$ and preserves all other leaves of $\SingFol$ in $\Zman$ with their orientations.

{\rm 2)}
Suppose in addition that $\Zman$ is orientable and let $\flow$ be any Hamiltonian like flow of $\func|_{\Zman}$ on $\Zman$.
Then there exists a unique $\Cinfty$ function $\alpha:\Zman\to\bR$ such that $\dif^2 = \flow_{\alpha}$ on $\Zman$ and $\alpha=0$ at each fixed point of $\dif^2$ in $\Int{\chp}$ and for each local extreme $\pz$ of $\func$ in $\Zman$.
\end{theorem}
\begin{proof}
1) Let us mention that since $\dif$ reverses orientation of $\Vman$, it reverses orientations of all regular leaves in $\chp$.
Therefore those leaves are $(\dif^2,+)$-invariant, and we should prove the same for all other leaves of $\SingFol$ in $\Zman$.

First we introduce the following notation.
If $\KK\cap\overline{\chpmm} \not= \varnothing$, then let $\Kmm$ be a unique critical leaf of $\func$ intersecting $\overline{\chpmm}$, and let $\Rmm$ be its $\func$-regular neighborhood.
Otherwise, when $\KK\cap\overline{\chpmm}= \varnothing$, put $\Kmm=\Rmm=\varnothing$,
Define further $\Kpp$ and $\Rpp$ in a similar way with respect to $\chppp$.
Then  \[ \Zman = \Rmm \cup \chpmm \cup \chppp \cup \Rpp. \]
As those four sets are invariant with respect to $\dif$, it suffices to prove that $\dif^2$ preserves leaves of $\SingFol$ with their orientation for each of those sets.

1a) Let us show that \myemph{$\dif^2$ preserves all leaves of $\SingFol$ in $\chpmm$.}

Since $\chp^{-}$ is an orientable manifold, one can construct a Hamiltonian like flow of $\func$ on $\chp^{-}$.
Evidently, $\flow$ satisfies conditions~\ref{cond:star:1}-\ref{cond:star:finite_comp}.
Moreover, since $\dif$ reverses orientation of all periodic orbits of $\flow$ in $\chppp$, we get from Theorem~\ref{th:shift_func_XS1_L} that there exists a unique $\Cinfty$ function $\alpha:\chp\to\bR$ such that $\dif^2|_{\chpmm} = \flow_{\alpha}$ and $\alpha$ vanishes at fixed points of $\dif^2$ on regular leaves of $\func$ in $\chpmm$.
In particular, each non-periodic orbit of $\flow$ in $\chppp$ is $(\dif^2, +)$-invariant as well.

1b) Now let us prove that each leaf of $\SingFol$ in $\Kmm$ is also $(\dif^2,+)$-invariant.
This will imply $(\dif^2,+)$-invariantness of all leaves of $\SingFol$ in $\Rmm$ (see the proof of the implication (iii)$\Rightarrow$(i) in~\cite[Lemma~7.4]{Maksymenko:TA:2020}).

If $\Kmm=\varnothing$ there is nothing to prove.

If $\Kmm \cap \chpmm = \varnothing$, then by Lemma~\ref{lm:chp_cyl_prop}\ref{enum:lm:chp_cyl_prop:3} $\Kmm$ is a local extreme of $\func$.
Hence $\Kman$ is an element of $\SingFol$ and its is evidently invariant with respect $\dif$, and therefore with respect to $\dif^2$.

Thus assume that $\Kmm \cap \chpmm \not= \varnothing$.
Then this intersection contains a non-periodic orbit $\gamma$ of $\flow$.
Since $\gamma$ is $(\dif^2,+)$-invariant, it follows from~\cite[Claim~7.1.1]{Maksymenko:AGAG:2006} or~\cite[Lemma~7.4]{Maksymenko:TA:2020}, that \myemph{all elements of the foliation $\SingFol$ are also $(\dif^2,+)$-invariant}.

Let us recall a simple proof of that fact.
Indeed, let $v$ be a vertex of $\gamma$ being therefore a critical point of $\func$.
Then $\dif(v)=v$, whence $\dif^2$ preserves the set of all edges incident to $v$.
Moreover, as $\dif^2$ preserves orientation at $v$, it must also \myemph{preserve cyclic order of edges incoming to $v$}.
But since $\gamma$ (being one of those edges) is $(\dif^2,+)$-invariant, it follows that so are all other edges incident to $v$.
Applying the same arguments to those edges and so on, we will see that $\dif^2$ preserves all edges of $\Kman$ with their orientation.

The proofs for $\chppp$ and $\Rpp$ are similar.

\smallskip

2) Assume now that $\Zman$ is an orientable surface and let $\flow$ be any Hamiltonian like flow of $\func$.
We know from 1) that $\dif^2$ preserves all orbits of $\flow$ with their orientations.

2a) We claim that \myemph{there exists a unique $\Cinfty$ function $\alpha:\chp\to\bR$ such that $\dif^2|_{\chp} = \flow_{\alpha}$ and $\alpha$ vanishes at fixed points of $\dif^2$ on periodic orbits.}

If both $\Kmm$ and $\Kpp$ are non-empty and distinct, then the restriction of $\flow$ to $\chp$ satisfies conditions~\ref{cond:star:1}-\ref{cond:star:finite_comp}.
Moreover, as $\dif$ reverses orientation of all periodic orbits of $\flow$, the statement follows from Theorem~\ref{th:shift_func_XS1_L} as in 1a).

However, if $\Kmm=\Kpp$, as in Example~\ref{exmp:chipped_cyl}d), the situation is slightly more complicated: $\chp$ might be not of the form $([-1,1]\times\Circle) \setminus \Qman$ for some finite set $\Qman \subset \{-1,1\}\times\Circle$ and Theorem~\ref{th:shift_func_XS1_L} is not directly applicable.
Nevertheless, one can apply that theorem to each of the sets $\chpmm$, $\chppp$, and $\Int{\chp}$ and construct three functions
\begin{align*}
  &\alpha^{-}:\chpmm\to\bR, &
  &\alpha^{-}:\chppp\to\bR, &
  &\alpha^{0}:\Int{\chp}\to\bR
\end{align*}
satisfying $\dif^2|_{\chpmm} = \flow_{\alpha^{-}}$, $\dif^2|_{\chppp} = \flow_{\alpha^{+}}$, $\dif^2|_{\Int{\chp}} = \flow_{\alpha^{0}}$, and vanishing at fixed points of $\dif^2$ on  periodic orbits.
From uniqueness of such functions, we get that $\alpha^{-} = \alpha^{0}$ on $\chpmm \cap \Int{\chp}$ and $\alpha^{+} = \alpha^{0}$ on $\chppp \cap \Int{\chp}$.

A possible problem is that $\chp$ intersects $\Kmm$ from ``both sides'', and therefore a priori $\alpha^{+}$ and $\alpha^{-}$ can differ on $\chpmm \cap \chppp \cap \Kmm$.
However, $\chpmm \cap \chppp \cap \Kmm$ consists of non-periodic orbits of $\flow$, and therefore foreach such orbit $\gamma$ the identity $\dif^2|_{\gamma} = \flow_{\alpha^{-}}|_{\gamma} = \flow_{\alpha^{+}}|_{\gamma}$ implies that $\alpha^{-} = \alpha^{+}$ on $\gamma$.

Thus $\alpha^{-} = \alpha^{+}$ on $\chpmm \cap \chppp \cap \Kmm$, and therefore those functions define a well defined $\Cinfty$ function $\alpha:\chp\to\bR$ satisfying $\dif^2|_{\Int{\chp}} = \flow_{\alpha}$ and $\alpha$ vanishes at fixed points of $\dif^2$ on periodic orbits.

2b) It remains to show that \myemph{$\alpha$ extends to a shift function for $\dif^2$ on $\Rmm\cup\Rpp$} and thus on all of $\Zman$.
It suffices to prove that for $\Rmm$.

If $\Kmm=\varnothing$, then $\Rmm=\varnothing$ and there is nothing to prove.

If $\Kmm$ is a local extreme of $\func$, then by Lemma~\ref{lm:ext_alpha} (cases \ref{emum:lm:ext_alpha:1} and~\ref{emum:lm:ext_alpha:3} for non-degenerate and degenerate critical point) $\alpha$ can be defined at $\Kmm$ so that it becomes $\Cinfty$.

In all other cases $\Kmm$ contains a non-periodic orbit of $\flow$.
Then by the implication (ii)$\Rightarrow$(iv) of \cite[Lemma~7.4]{Maksymenko:TA:2020}, $\alpha$ extends to a $\Cinfty$ shift function for $\dif^2$ on $\Rmm$.
It remains to prove the following lemma:
\begin{sublemma}
$\alpha(\pz)=0$ for every local extreme of $\func$ in $\Zman$.
\end{sublemma}
\begin{proof}
Indeed, it is evident, that arbitrary small neighborhood of $\pz$ contains a periodic orbit $\gamma$ of $\flow$.
Since $\dif$ reverses orientations of $\gamma$, we have from Lemma~\ref{lm:lift_dif_circle}\ref{enum:lm:lift_dif_circle:fixed_pt} that $\dif$ always has at least one fixed point $\px \in \gamma$ (in fact it has even two such points).
Hence by Corollary~\ref{cor:shift_func_circle}\ref{enum:cor:shift_func_circle:2}, $\alpha(\px)=0$.
Then by continuity of $\alpha$ we should have that $\alpha(\pz)=0$ as well.
\end{proof}

Theorem~\ref{th:shift_func_on_chip_nbh} is completed.
\end{proof}

\section{Creating almost periodic diffeomorphisms}\label{sect:change_dif}
Let $\Mman$ be a compact orientable surface, $\func\in\FMP$, $\Zman$ be an $\func$-adapted subsurface, $\dif\in\Stabilizer{\func}$ be such that $\dif(\Zman) = \Zman$, and $m\geq2$.
If $\dif^{m}|_{\Zman}$ is isotopic to the identity of $\Zman$ by $\func$-preserving isotopy, then the following Lemma~\ref{lm:gm_in_Sid} gives conditions when one can change $\dif$ on $\Mman\setminus\Zman$ so that its $m$-power will be $\func$-preserving isotopic to the identity on all of $\Mman$.
The proof follows the line of~\cite[Lemma~13.1(3)]{Maksymenko:TA:2020} in which $\Mman$ is a $2$-disk or a cylinder.
\begin{lemma}\label{lm:gm_in_Sid}{\rm(cf.~\cite[Lemma~13.1(3)]{Maksymenko:TA:2020})}
Let $\Mman$ be a connected compact orientable surface, $\func\in\FMP$, $\Zman$ be an $\func$-adapted subsurface, and $\dif\in\Stabilizer{\func}$ be such that $\dif(\Zman) = \Zman$.
Suppose that the following conditions hold.
\begin{enumerate}[leftmargin=*, label={\rm(\arabic*)}]
\item\label{enum:lm:gm_in_Sid:1}
Each component of $\Zman$ contains at least one saddle point of $\func$.

\item\label{enum:lm:gm_in_Sid:3}
$\dif^m$ is isotopic in $\Stabilizer{\func}$ to a diffeomorphism $\tau$ fixed on some neighborhood of $\Zman$ {\rm(}by \cite[Lemma~7.1]{Maksymenko:TA:2020} this condition holds if there exists a $\Cinfty$ function $\alpha:~\Zman\to\bR$ such that $\dif^m|_{\Zman} = \flow_{\alpha}${\rm)}.

\item\label{enum:lm:gm_in_Sid:2}
There exists $m\geq2$ and $a\geq1$ such that connected components of $\overline{\Mman\setminus\Zman}$ can be enumerated as follows:
\begin{equation}\label{equ:Yij}
\begin{array}{cccc}
\Yman_{1,0}   & \Yman_{1,1}   & \cdots & \Yman_{1,m-1}   \\
\Yman_{2,0}   & \Yman_{2,1}   & \cdots & \Yman_{2,m-1}   \\
\cdots        & \cdots        & \cdots & \cdots        \\
\Yman_{a,0}   & \Yman_{a,1}   & \cdots & \Yman_{a,m-1}
\end{array}
\end{equation}
so that $\dif(\Yman_{i,j})=\Yman_{i,j+1\bmod m}$ for all $i,j$, that is $\dif$ cyclically shifts columns in~\eqref{equ:Yij}.
\end{enumerate}

Then there exists $\gdif\in\Stabilizer{\func}$ such that $\gdif=\dif$ on $\Zman$ and $\gdif^{m} \in\StabilizerId{\func}$, that is $\gdif^{m} = \flow_{\beta}$ for some $\Cinfty$ function $\beta:\Mman\to\bR$.
\end{lemma}
\begin{proof}
Let $\Yman_j = \bigcup\limits_{i=1}^{a}\Yman_{i,j}$, $j=0,\ldots,m-1$, be the union of components from the same column of~\eqref{equ:Yij}.
Then $\dif(\Yman_j) = \Yman_{j+1\bmod m}$.
Notice that condition~\ref{enum:lm:gm_in_Sid:2} implies that $\dif^{m}(\Yman_{i,j}) = \Yman_{i,j}$ for all $i,j$.

We will show that the the desired diffeomorphism $\gdif\in\Stabilizer{\func}$ can be defined by the formula:
\[
\gdif(\px) =
\begin{cases}
\dif(\px), & \px \in \Zman \cup \Yman_0 \cup \cdots \cup \Yman_{m-2}, \\
\tau^{-1}\circ\dif(\px), & \px \in \Yman_{m-1}.
\end{cases}
\]
Indeed, by definition $\gdif = \dif$ on $\Zman$.
Moreover, as $\tau$ is fixed on some neighborhood of $\Zman$, it also fixed near $\Zman\cap\Yman_{m-1}$.
Therefore $\gdif = \dif = \tau^{-1}\circ\dif$ near $\Zman\cap\Yman_{m-1}$, and so $\gdif$ is a well defined $\Cinfty$ map.
It remains to prove that $\gdif^m = \flow_{\beta} \in \StabilizerId{\func}$ for some $\Cinfty$ function $\beta:\Mman\to\bR$.

Let $\flow$ be a Hamiltonian flow for $\func$.
Since $\tau$ and $\dif^m$ are isotopic in $\Stabilizer{\func}$, it follows that $\tau^{-1}\circ\dif^m\in\StabilizerId{\func}$.
Hence by Lemma~\ref{lm:char_Sid}, $\tau^{-1}\circ\dif^{m} = \flow_{\alpha}$ for some $\Cinfty$ function $\alpha:\Mman\to\bR$.

Since $\StabilizerId{\func}$ is a normal subgroup of $\Stabilizer{\func}$, it follows that
\[
\dif^{j} \circ (\tau^{-1}\circ\dif^{m}) \circ \dif^{-j} =  \dif^{j} \circ \tau^{-1}\circ\dif^{m-j} \in \StabilizerId{\func},
\qquad j=0,\ldots,m-1,
\]
as well.
Therefore, again by Lemma~\ref{lm:char_Sid}, $\dif^{j} \circ \tau^{-1}\circ\dif^{m-j} = \flow_{\alpha_j}$ for some $\Cinfty$ function $\alpha_j:\Mman\to\bR$.

As $\tau$ is fixed on some neighborhood of $\Zman$, it follows that for each $j$
\[
  \flow_{\alpha_{j}}  = \dif^{j} \circ \tau^{-1}\circ\dif^{m-j} = \tau^{-1}\circ\dif^{m} = \flow_{\alpha}
  \ \text{on} \ \Zman.
\]
Then the assumption~\ref{enum:lm:gm_in_Sid:1} that every connected component $\Zman'$ of $\Zman$ contains a saddle point, implies that $\flow$ has a non-closed orbit $\gamma$ in $\Zman'$.
Therefore $\alpha=\alpha_j$ on $\gamma$.
Since $\Zman'$ is connected, it follows from local uniqueness of shift functions for $\tau^{-1}\circ\dif^{m}|_{\Zman}$ (see Corollary~\ref{cor:local_uniqueness}) that $\alpha=\alpha_j$ near $\Zman'$ for all $j=0,\ldots,m-1$.
Hence $\alpha=\alpha_j$ near all of $\Zman$ for all $j=0,\ldots,m-1$.

Thus we obtain a well-defined $\Cinfty$ function $\beta:\Mman\to\bR$ given by:
\[
\beta(\px) =
  \begin{cases}
    \alpha(\px), & \px\in\Zman, \\
    \alpha_j(\px), & \px\in\Yman_j, \ j=0,\ldots,m-1.
  \end{cases}
\]

We claim that $\gdif^m = \flow_{\beta}$.

a) Indeed, if $\px\in\Zman$, then $\gdif^m(\px) = \dif^m(\px) = \flow_{\alpha}(\px) = \flow_{\beta}(\px)$.

b) Also notice that $\gdif(\Yman_{i,j})=\Yman_{i,j+1\bmod m}$ and $\gdif(\Yman_j)=\Yman_{j+1\bmod m}$.
Then $\gdif^m|_{\Yman_j}$ is the following composition of maps:
\[
\Yman_j \ \xrightarrow{~\dif~} \ \Yman_{j+1} \ \xrightarrow{~\dif~} \
 \cdots \ \xrightarrow{~\dif~} \ \Yman_{m-1}        \ \xrightarrow{~\tau^{-1}\circ\dif~}  \
 \Yman_{0} \ \xrightarrow{~\dif~} \ \cdots \ \xrightarrow{~\dif~} \ \Yman_j.
\]
which thus coincides with $\dif^{j} \circ \tau^{-1} \circ \dif^{m-j} = \flow_{\alpha_{j}} = \flow_{\beta}$.
\end{proof}

\section{Proof of Theorem~\ref{th:exist:g_rev_or}}\label{sect:proof:th:exist:g_rev_or}
Let $\Mman$ be a connected compact surface, $\func\in\FMP$, $\dif\in\Stabilizer{\func}$, $\Aman$ be the union of all regular leaves of $\func$ being $\minv{\dif}$-invariant and $\Kman_1,\ldots,\Kman_{\vk}$ be all the critical leaves of $\func$ such that $\overline{\Aman}\cap\Kman_i\not=\varnothing$.
For $i=1,\ldots,\vk$, let $\regN{\Kman_i}$ be an $\func$-regular neighborhood of $\Kman_i$ chosen so that $\regN{\Kman_i} \cap \regN{\Kman_j} = \varnothing$ for $i\not= j$ and
\[
   \Zman := \Aman \bigcup  \Bigl( \mathop{\cup}\limits_{i=1}^{\vk} \regN{\Kman_i} \Bigr).
\]
Assume that $\Zman$ is non-empty, orientable and every connected component $\gamma$ of $\partial\Zman \cap \Int{\Mman}$ separates $\Mman$.
We have to prove that there exists $\gdif\in\Stabilizer{\func}$ which coincide with $\dif$ on $\Zman$ and such that $\gdif^2 \in \StabilizerId{\func}$.

\begin{lemma}\label{shift}
There exists a unique $\Cinfty$ function $\alpha:\Zman\to\bR$ such that $\dif^2|_{\Zman} = \flow_{\alpha}$ and $\alpha=0$ at each fixed point of $\dif^2$ on $\minv{\dif}$-invariant regular leaves of $\func$.
\end{lemma}
\begin{proof}
Let $\Vman$ be a regular leaf $\Vman$ of $\func$, and $\chp$ a chipped cylinder of $\func$ such that $\Vman\subset\Int{\chp}$.
If $\Vman$ is $\minv{\dif}$-invariant, then so is every other regular leaf $\Vman'\subset\Int{\chp}$.
This implies that $\Zman$ is a union of $\func$-regular neighborhoods $\Zman_1,\ldots,\Zman_{\vl}$ of some chipped cylinders $\chp_1,\ldots,\chp_{\vl}$ of $\func$.

By Theorem~\ref{th:shift_func_on_chip_nbh}, for each $i=1,\ldots,\vl$ there exists a unique $\Cinfty$ function $\alpha_i:\Zman_i\to\bR$ such that $\dif^2|_{\Zman_i} = \flow_{\alpha_i}$ and $\alpha=0$ at each fixed point of $\dif^2$ on each on $\minv{\dif}$-invariant regular leaf of $\func$ in $\Zman_i$.

Notice that if $\Zman_i \cap \Zman_j \not=\varnothing$, then every connected component $\Wman$ of that intersection always contains a non-periodic orbit $\gamma$ of $\flow$.
Therefore by uniqueness of shift functions (Corollary~\ref{cor:local_uniqueness}) we obtain that $\alpha_i=\alpha_j$ on $\Wman$.

Hence the functions $\{ \alpha_i \}_{i=1,\ldots,\vl}$ agree on the corresponding intersections, and therefore they define a unique $\Cinfty$ function $\alpha:\Zman\to\bR$ such that $\dif^2|_{\Zman}=\flow_{\alpha}$.
Then $\alpha=0$ at each on $\minv{\dif}$-invariant regular leaf of $\func$ in $\Zman_i$.
\end{proof}
If $\Mman=\Zman$, then theorem is proved.
Thus suppose that $\Mman\not=\Zman$.
\begin{lemma}\label{pairs}
The number of connected components $\overline{\Mman\setminus\Zman}$ is \myemph{even}, and they can be enumerated by pairs of numbers:
\begin{equation}\label{equ:Yij:2}
\begin{aligned}
\Yman_{1,0} \quad \Yman_{2,0} \quad \ldots \quad \Yman_{a,0} \\
\Yman_{1,1} \quad \Yman_{2,1} \quad \ldots \quad \Yman_{a,1}
\end{aligned}
\end{equation}
for some $a>1$ so that $\dif$ exchanges the rows in~\eqref{equ:Yij:2}, that is $\dif(\Yman_{i,0})=\Yman_{i,1}$ and $\dif(\Yman_{i,1})=\Yman_{i,0}$ for each $i$.
\end{lemma}
\begin{proof}
Let $\Yman_1,\ldots,\Yman_q$ be all the connected components of $\overline{\Mman\setminus\Zman}$.
Denote $\gamma_i := \Yman_i \cap \Zman$.
Then by condition~\ref{cond:B}, $\gamma_i$ is a unique common boundary component of $\Yman_i$ and $\Zman$.

Since $\dif(\Zman)=\Zman$, it follows that $\dif$ induces a permutation of connected components of $\{\Yman_i\}_{i=1,\ldots,q}$.
Moreover, by Lemma~\ref{lm:dZ_prop} $\dif(\gamma_i) \cap \gamma_i = \varnothing$, whence $\dif(\Yman_i) \cap \Yman_i = \varnothing$ as well.
On the other hand, $\dif^2 = \flow_{\alpha}$, whence $\dif^2(\gamma_i) = \flow_{\alpha}(\gamma_i) = \gamma_i$, and therefore $\dif^2(\Yman_i) =\Yman_i$.

Thus $\{\Yman_i\}_{i=1,\ldots,q}$ splits into pairs which are exchanged by $\dif$.
\end{proof}

Now it is enough to apply Lemma~\ref{lm:gm_in_Sid} with $m=2$. Then there exists $\gdif$ such that $\gdif=\dif$ on $\Zman$ and $\gdif^2 \in\StabilizerId{\func}$. Notice that each component of $\Zman$ contains at least one saddle point of $\func$, otherwise by Lemma \ref{lm:chp_cyl_prop} \ref{saddle_p} we have that $M=\Zman$.
So the first condition \ref{enum:lm:gm_in_Sid:1} of Lemma~\ref{lm:gm_in_Sid} holds. The second condition \ref{enum:lm:gm_in_Sid:3} follows from Lemma \ref{shift} and the third condition \ref{enum:lm:gm_in_Sid:2} follows from Lemma \ref{pairs}.
Theorem~\ref{th:exist:g_rev_or} is completed.
\qed

\bibliographystyle{pigc_plain}
\bibliography{disk_ch_or}

\def\cprime{$'$}
\begin{thebibliography}{1}

\bibitem{Adelson-Welsky-Kronrode:DAN:1945}
G.~M. Adelson-Welsky, A.~S. Kronrode.
\newblock Sur les lignes de niveau des fonctions continues poss\'{e}dant des
  d\'{e}riv\'{e}es partielles.
\newblock {\em C. R. (Doklady) Acad. Sci. URSS (N.S.)}, 49:235--237, 1945.

\bibitem{Franks:Top:1985}
John Franks.
\newblock Nonsingular {S}male flows on {$S^3$}.
\newblock {\em Topology}, 24(3):265--282, 1985,
  doi:~\href{http://dx.doi.org/10.1016/0040-9383(85)90002-3}{10.1016/0040-9383(85)90002-3}.

\bibitem{Kronrod:UMN:1950}
A.~S. Kronrod.
\newblock On functions of two variables.
\newblock {\em Uspehi Matem. Nauk (N.S.)}, 5(1(35)):24--134, 1950.

\bibitem{Maksymenko:TA:2003}
S.~I. Maksymenko.
\newblock Smooth shifts along trajectories of flows.
\newblock {\em Topology Appl.}, 130(2):183--204, 2003.

\bibitem{Maksymenko:hamv2}
S.~I. Maksymenko.
\newblock Hamiltonian vector fields of homogeneous polynomials in two
  variables.
\newblock {\em Pr. Inst. Mat. Nats. Akad. Nauk Ukr. Mat. Zastos.},
  3(3):269--308, arXiv:math/0709.2511, 2006.

\bibitem{Maksymenko:AGAG:2006}
S.~I. Maksymenko.
\newblock Homotopy types of stabilizers and orbits of {M}orse functions on
  surfaces.
\newblock {\em Ann. Global Anal. Geom.}, 29(3):241--285, 2006.

\bibitem{Maksymenko:ProcIM:ENG:2010}
S.~I. Maksymenko.
\newblock Functions with isolated singularities on surfaces.
\newblock {\em Geometry and topology of functions on manifolds. Pr. Inst. Mat.
  Nats. Akad. Nauk Ukr. Mat. Zastos.}, 7(4):7--66, 2010.

\bibitem{Maksymenko:TA:2020}
Sergiy Maksymenko.
\newblock Deformations of functions on surfaces by isotopic to the identity
  diffeomorphisms.
\newblock {\em Topology Appl.}, 282:107312, 48, 2020,
  doi:~\href{http://dx.doi.org/10.1016/j.topol.2020.107312}{10.1016/j.topol.2020.107312}.

\bibitem{Reeb:CR:1946}
Georges Reeb.
\newblock Sur les points singuliers d'une forme de {P}faff compl\`etement
  int\'egrable ou d'une fonction num\'erique.
\newblock {\em C. R. Acad. Sci. Paris}, 222:847--849, 1946.

\end{thebibliography}
\printArticleAuthorsInfo{\thearticlesnum}

\end{document}